\documentclass{sn-jnl}

\usepackage{multirow}%
\usepackage{amsmath,amssymb,amsfonts}%
\usepackage{amsthm}%
\usepackage{mathrsfs}%
\usepackage[title]{appendix}%
\usepackage{xcolor}%
\usepackage{textcomp}%
\usepackage{manyfoot}%
\usepackage{booktabs}%
\usepackage{algorithm}%
\usepackage{algorithmicx}%
\usepackage{algpseudocode}%
\usepackage{listings}%
\usepackage{graphicx}
\usepackage{bookmark}
\usepackage{bm}
\usepackage{amsmath,cases}
\usepackage{algorithm}
\usepackage{multirow}
\usepackage{longtable}
\usepackage{cleveref}
\usepackage{graphics,epsfig,subfigure}
\usepackage{color}
\usepackage[numbers]{natbib}
\usepackage{extarrows}
\usepackage{turnstile}
\usepackage{fancyhdr}
\graphicspath{{jpg/}}

\usepackage{lipsum}
\usepackage{amsfonts}
\usepackage{graphicx}
\usepackage{epstopdf}

\theoremstyle{thmstyleone}%
\newtheorem{theorem}{Theorem}
\theoremstyle{thmstyletwo}%
\newtheorem{example}{Example}%
\newtheorem{assumption}{Assumption}
\newtheorem{lemma}{Lemma}

\theoremstyle{thmstylethree}%

\newcommand{\dd}{\,{\rm d}}

\newcommand{\dist}{\,{\rm dist}}
\newcommand{\divo}{\,{\rm div}}

\begin{document}

\title[Article Title]{Optimal error estimates of the diffuse domain method for second order parabolic equations}


\author[1]{\fnm{Wenrui} \sur{Hao}}\email{wxh64@psu.edu}
\equalcont{These authors contributed equally to this work.}

\author[2]{\fnm{Lili} \sur{Ju}}\email{ju@math.sc.edu}
\equalcont{These authors contributed equally to this work.}

\author*[1]{\fnm{Yuejin} \sur{Xu}}\email{ymx5204@psu.edu}

\affil*[1]{\orgdiv{Department of mathematics}, \orgname{Pennsylvania State University}, \orgaddress{\city{College Park}, \postcode{16802}, \state{PA}, \country{USA}}}

\affil[2]{\orgdiv{Department of Mathematics}, \orgname{University of South Carolina}, \orgaddress{\city{Columbia}, \postcode{29208}, \state{Columbia}, \country{USA}}}


\abstract{In this paper, we study the convergence behavior of the diffuse domain method (DDM) for solving a class of second-order parabolic partial differential equations with Neumann boundary condition posed on general irregular domains. The DDM employs a phase-field function to extend the original parabolic problem to a similar but slightly modified problem defined over a larger rectangular domain that contains the target physical domain. Based on the weighted Sobolev spaces, we rigorously establish the convergence of the diffuse domain solution to the original solution as the interface thickness parameter goes to zero, together with the corresponding optimal error estimates under the weighted $L^2$ and $H^1$ norms. Numerical experiments are also presented to validate the theoretical results.}




\maketitle
\section{Introduction}

Combined with appropriate initial values and boundary conditions, parabolic partial differential equations (PDEs) have been widely applied in various mathematical models. Notable examples include the Navier-Stokes equations and the time-dependent advection-diffusion equation in fluid dynamics \cite{Roger1984}, Stokes-Darcy problems arising in petroleum and biomedical engineering \cite{MartinaBoris2023}, as well as the Allen-Cahn equation and other phase field models used to describe phase transitions and separations \cite{DuFeng2020,Acta1979}, among many others. Most existing numerical methods for solving interface problems are based on sharp interface approaches, which rely on explicit surface parameterization. This requirement poses a significant challenge for handling complex geometries and mesh generation. Examples of such methods include the extended and composite finite element methods \cite{DolbowHarari2009,FriesBelytschko2010}, immersed interface methods \cite{LeVequeLi1995,LiIto2006,WangKai2023}, virtual node methods with embedded boundary conditions \cite{BedrossianZhu2010,HellrungLee2012}, and matched interface and boundary methods \cite{ZhaoWei2009,LiFeng2015,AmlanRay2023}, among others. Many of these techniques require specialized computational tools that are not readily available in standard finite element and finite difference software packages.

Over the past two decades, the diffuse interface method (DIM) has gained widespread attention and is regarded as an effective alternative approach to sharp interface methods for solving PDE problems on complex geometries. This method represents the physical domain implicitly using a phase-field function, which can be interpreted as the domain's indicator function when the interface thickness approaches zero and  is often coupled with classic discretization schemes, such as finite difference method,  finite element method \cite{YangMao2019}, spectral method \cite{BuenoFenton2006}, and Nitsche's method \cite{NguyenStoter2018}. 
The DIM  has been applied to elliptic interface problems, two-phase flow problems, and various material and engineering applications. Kockelkoren et al. \cite{KockelkorenLevine2003} were the first to apply the diffuse interface method to study diffusion inside a cell with zero Neumann boundary conditions.  Many numerical methods have been developed for solving different two-phase flow problems based on DIM. For example, Liu et al. combined the DIM with the consistent and conservative phase-field method to solve two-phase flows in complex geometries \cite{LiuChai2022}. The DIM was applied to solve two-phase flows of viscous, incompressible fluids with matched densities, leading to coupled Navier-Stokes or Cahn-Hilliard systems \cite{FrigeriGrasselli2015,Abels2009}. Moreover, the diffuse interface method can address miscible fluids of different densities \cite{AbelsLengeler2014}, compressible fluids \cite{FeireislRocca2010}, Stokes-Darcy coupled equations \cite{MartinaBoris2023}, and problems involving more than two phases, where additional labeling functions are introduced to distinguish among them \cite{BrannickLiu2015}. The diffuse interface method also has been extensively used to solve various models, such as the material model with interfaces that can be advected or stretched \cite{TeigenLi2009,TorabiLowengrub2009}, the patient-specific human liver model based on MRI scans \cite{Stoter2017}, the PDEs in moving geometries \cite{ElliottStinner2011}.The DIM approach is also a valuable tool for solving variational inverse problems, with corresponding convergence rates estimated in \cite{BurgerElvetun2015}. 

The diffuse domain method (DDM) \cite{AndersonMcFadden1998,AndreasAxel2006}, can be viewed as a variant of  DIM. 
The DDM represents the original physical domain implicitly using a phase-field function with a narrow diffuse interface layer, wherein the value of the phase-field function rapidly transits from 1 inside the domain to 0 outside the domain \cite{KarlLowengrub2015}. The original PDE problem is then subsequently reformulated into a similar problem in a larger rectangular domain. As a result, the challenges associated with mesh generation for complex domains are mitigated, enabling the straightforward generation of spatial meshes for the rectangular domain to solve the transformed PDE problem with existing numerical schemes. 
The DDM has been used to solve PDEs in complex, stationary, or moving geometries with Dirichlet, Neumann, and Robin boundary conditions \cite{LiLowengrub2009}. 
The properties and convergence of the diffuse domain method have  been analyzed in several studies. Li et al. demonstrated that, in the diffuse domain method, several approximations to the physical boundary conditions converge asymptotically to the correct sharp interface problem \cite{LiLowengrub2009}. They also observed that the choice of boundary condition can significantly affect numerical accuracy. Furthermore, Lervag et al. discussed that for specific choices of boundary condition approximations, the asymptotic convergence of the diffuse domain method can be improved to second order \cite{KarlLowengrub2015}. Lowengrub et al. have applied this method to elliptic problems and provided an asymptotic analysis of the boundary layer in \cite{KarlLowengrub2015, AlandLowengrub2010}. Additionally, Franz et al. analyzed the error estimates in the $L^{\infty}$-norm for one-dimensional elliptic equations \cite{FranzRoos2012}. Numerical errors in $L^2$, $L^{\infty}$, and $H^1$-norms on the original region have been further investigated for elliptic problems with Dirichlet boundary conditions \cite{Schlottbom2016}. Burger et al. constructed weighted Sobolev spaces based on the phase-field function and analyzed the approximation error in the extended region within the weighted $L^2$-norm \cite{BurgerElvetun2017}. Moreover, the convergence rate on the original domain for the Stokes-Darcy coupled problem has been discussed in \cite{MartinaBoris2023}. Guo et al. coupled the DDM with an interface model to simulate two-phase fluid flows with variable physical properties while maintaining thermodynamic consistency \cite{GuoYu2021}. 
The DDM approach was also  employed to develop biomedical models, such as the chemotaxis-fluid diffuse-domain model for simulating bioconvection \cite{WangChertock2023}, a needle insertion model \cite{Jerg2020}, etc.

In this paper, we are concerned with  the diffuse domain method for solving the following linear second-order parabolic equation with a Neumann boundary condition:
\begin{equation}
	\label{eq2-1}
	\left\{
	\begin{aligned}
		& u_t = \nabla\cdot(A\nabla u) + f(t), \quad && \bm x \in D,\quad 0 \leq t \leq T\\
		& u|_{t=0} = u_0, \quad && \bm x \in D,\\
		& (A\nabla u) \cdot \bm n = g(t), \quad && \bm x \in \partial D, \quad 0 \leq t \leq T,
	\end{aligned}
	\right.
\end{equation}
where $D$ is an open bounded Lipschitz domain with irregular shape in $\mathbb{R}^d$ ($d \geq 1$), $T>0$ is the terminal time, $A(\bm x)>0$ is the diffusion coefficient and fulfills $\kappa\leq A(\bm x)\leq \kappa^{-1}$ for all $\bm x \in D$ with some constant $\kappa>0$, $u(t,\bm x)$ is the unknown function,   $f(t,\bm x)$ is the source term, $u_0(\bm x)$ is the initial value, and $g(t,\bm x)$ is the Neumann boundary value. We will analyze the convergence of the diffuse domain solution as the interface thickness parameter goes to zero and derive corresponding error estimates  measured in the weighted $L^2$ and weighted $H^1$ norms.  The  analysis techniques  mainly follow the weighted Sobolev space-based  framework developed in \cite{BurgerElvetun2017}, but also  with some enhancements. Our results successfully decouple the relationship between the hidden constants and the interface thickness parameter  in  the error estimates. Furthermore, the convergence rates are also improved to second order in the weighted $L^2$ norm and first order in the weighted $H^1$ norm, which are optimal  as verified by the numerical experiments. To the best of our knowledge, this work presented in this paper is the first study on rigorous error analysis of the DDM for second-order parabolic equations.

The rest of the paper is organized as follows. The diffuse domain method for the model problem \eqref{eq2-1} is first described  in Section \ref{algorithm-description}, and several preliminaries and lemmas for the weighted Sobolev spaces are given in Section \ref{pre}. The convergence of the diffuse domain solution to the original solution as the interface thickness parameter goes to zero is proved in  Section \ref{theory}, together with the corresponding optimal error estimates under the weighted $L^2$ and $H^1$ norms. In Section \ref{numerical},  some numerical experiments are carried out to verify the theoretical results. Finally, some concluding remarks are drawn in Section \ref{conclusion}.

\section{The diffuse domain method}\label{algorithm-description}

First of all, some standard notations are proposed for later provement. For a given open bounded Lipschitz domain $\Omega \subset \mathbb{R}^d$ and for nonnegative integer $s$, denote $H^s(\Omega)$ as the standard integer-order Sobolev spaces on $\Omega$ with norm $\|\cdot\|_{s,\Omega}$ and semi-norm $|\cdot|_{s,\Omega}$, and the corresponding $L^2$-inner product is $(\cdot,\cdot)_{\Omega}$. The corresponding norm of space $H^s(\Omega)$ is $\|\cdot\|_{s,\Omega}$ and $\|v\|_{k,\infty,\Omega}={\rm ess}\sup_{|\bm \alpha|\leq k}\|D^{\bm \alpha}v\|_{L^{\infty}(\Omega)}$ for any function $v$ such that the right-hand side term makes sense, where $\bm \alpha=(\alpha_1, \cdots, \alpha_d)$ is a multi-index and $|\bm \alpha|=\alpha_1+\cdots+\alpha_d$. $H_0^s(\Omega)$ is the closure of $C_0^{\infty}(\Omega)$ with homogeneous Dirichlet boundary conditions. 
Moreover, given two quantities $a$ and $b$, $a \lesssim b$ is the abbreviation of $a \leq Cb$, where the hidden constant $C$ is positive and independent of the mesh size; $a\eqsim b$ is equivalent to $a\lesssim b \lesssim a$.


Assume that $u_0\in H^2(D)$, $A\in H^1(D)$, $g(t)\in L^2(0,T;L^2(\partial D))$ and $f(t)\in L^2(0,T;L^2(D))$, then
the variational formulation of \eqref{eq2-1} is given by: find $u \in L^2(0,T;H^1(D))$ and $u_t \in L^2(0,T;L^2(D))$ such that
\begin{equation}
	\label{eq2-2}
	\left\{\begin{aligned}
		&(u_t,v)+a(u,v)=\ell(v), \qquad \forall \,v \in H^1(D),\ 0 \leq t \leq T,\\
		&u(0,\bm x)=u_0(\bm x),
	\end{aligned}
	\right.
\end{equation}
where the bilinear operator $a(\cdot,\cdot)$ is symmetric positive and defined by
\begin{equation}
	\label{bilinear_operator}
	a(w,v)=\int_{D} A\nabla w\cdot\nabla v \dd \bm x, \qquad \forall\,w, v\in H^1(D),
\end{equation}
and the linear operator $\ell(\cdot)$ is defined by
\begin{equation}
	\label{linear_operator}
	\ell(v)=\int_{D} fv\dd \bm x + \int_{\partial D} gv\dd \sigma,\qquad \forall\,  v\in H^1(D),
\end{equation}

Next, we apply the DDM to approximate the above integrals on the domain $D$ \cite{BurgerElvetun2017}. 
First, let us introduce a signed distance function $d_D(\bm x)=\dist(\bm x,D)-\dist(\bm x,\mathbb{R}^n\setminus D)$, $\bm x \in \mathbb{R}^n$. It's obvious that the domain $D$ can be represented as $D=\{\bm x\,|\,d_D(\bm x)<0 \}$ with  $\partial D=\{\bm x\,|\,d_D(\bm x)=0 \}$. To relax the sharp interface representation, let further  introduce $\varphi^{\epsilon}(\bm x)=S(-d_D(\bm x)/\epsilon)$, where $\epsilon>0$ is a small interface thickness parameter and $S$ being a smooth function such  that $S(s) = -1$ for $s< -1$, 
$S(s) = 1$ for $s> 1$, and monotonic transition occurs when $s\in (-1,1)$. 
For instance, the sigmoidal  function \begin{equation} S(s) = \tanh(3s) \end{equation} is often taken  and we follow it in this paper.  As $\epsilon$ tends to zero, $S(\cdot/\epsilon)$ converges to the sign function, and hence, the phase-field function $$\omega^{\epsilon}(\bm x):=(1+\varphi^{\epsilon}(\bm x))/2$$  converges to the indicator function $\chi_D$ of $D$. 
It is easy to find that $\omega_{\epsilon}=0$ for $\bm x \notin D_{\epsilon}$, $0\leq\omega_{\epsilon}\leq \frac 1 2$ for $\bm x \in D_{\epsilon}\setminus D$, $\frac 1 2\leq\omega_{\epsilon}\leq 1$ for $\bm x \in D\setminus D_{-\epsilon}$ and $\omega_{\epsilon}=1$ for $\bm x \in D_{-\epsilon}$. Furthermore, we can easily find that $|S'(s)|\lesssim 1$ for all $s$, and \begin{equation}\int_{-\epsilon}^{\epsilon} \frac{3}{2\epsilon}S'\left(-\frac{3s}{\epsilon}\right)\dd s=1.\end{equation}
The key idea of DDM  is to use a weighted averaging of the integrals over $D_s=\{\bm x\,|\,d_D(\bm x)<s \}$, $s \in (-\epsilon,\epsilon)$, instead of integrating over the original irregular domain $D=D_0$ only. In order to generate the spatial mesh  conveniently, one usually further fixes a larger rectangular domain $\Omega$ such that $D\subset D_{\epsilon} \subset \Omega$ in practice. Fig. \ref{fig1} describe the geometric relation among the original physical domain $D$, the $\epsilon$-extension domain $D_{\epsilon}$ and the covering rectangular domain $\Omega$. Since $\omega_{\epsilon}=0$ for $\bm x \notin D_{\epsilon}$, we can compute the weighted averaging of the integrals over regular domain $\Omega$.

\begin{figure}[htbp]
	\centering
	\label{fig1}
	\centering
	\includegraphics[width = 140pt,height=140pt]{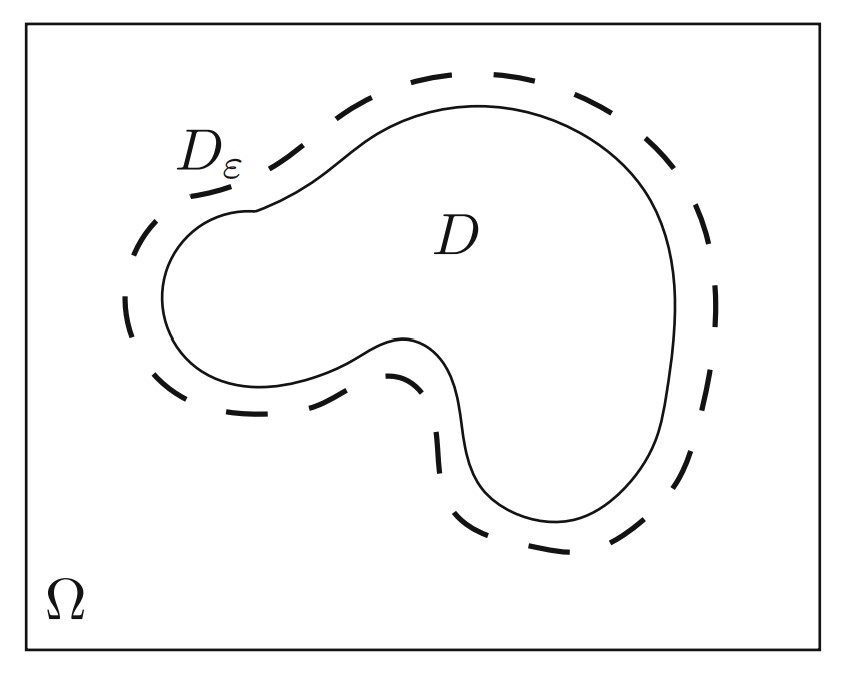}
	\caption{The relationship among $D$, $D_{\epsilon}$ and $\Omega$ for some $\epsilon>0$.}
\end{figure}

Following the similar arguments of \cite{BurgerElvetun2017}, consider the approximation of boundary integration over $\partial D$. Since $\frac{3}{2\epsilon} S'(\frac{3s}{\epsilon})$ approximates a concentrated distribution at zero, we expect for any integrable function $h:\Omega \to \mathbb{R}$,
\begin{equation*}
	\begin{split}
		\int_D h(\bm x)\dd \bm x &=\int_{-\infty}^{\infty} \frac 3 {2\epsilon} S'\left(-\frac{3s}{\epsilon}\right)\int_{D_0} h(\bm x)\dd \bm x\dd s\\
		&\approx \int_{-\epsilon}^{\epsilon} \frac{3}{2\epsilon} S'\left(-\frac{3s}{\epsilon}\right)\int_{D_s} h(\bm x)\dd \bm x \dd s\\
		&=\frac{1}{2}\int_{-1}^{1}\int_{\{\varphi^{\epsilon}>s\}} h(\bm x)\dd \bm x \dd s.
	\end{split}
\end{equation*}
Using the Fubini's theorem, we may further rewrite
\begin{equation}\label{approx1}
	\int_{-1}^1 \int_{\{\varphi^{\epsilon}>s \}}h(\bm x)\dd \bm x\dd s = \int_{D_{\epsilon}}\int_{-1}^{\varphi^{\epsilon}(\bm x)}\dd s\ h(\bm x)\dd \bm x = \int_{D_{\epsilon}}(1+\varphi^{\epsilon}(\bm x))h(\bm x)\dd \bm x.
\end{equation}
By the co-area formula, we then derive an approximation for the boundary integral as
\begin{equation}\label{approx2}
	\begin{aligned}
		\int_{\partial D} h(\bm x)\dd \sigma(\bm x) \approx &\; \frac12\int_{-1}^1\int_{\partial\{\varphi^{\epsilon}>s \}}h(\bm x)\dd \sigma(\bm x)\dd s\\
		=&\;\frac12\int_{D_{\epsilon}} h(\bm x)\left|\nabla \varphi^{\epsilon}(\bm x)\right|\dd \bm x\\
		=&\int_{D_{\epsilon}} h(\bm x)\left|\nabla \omega^{\epsilon}(\bm x)\right|\dd \bm x.
	\end{aligned}
\end{equation}

Let us define the weighted  $L^p$ space on $D_{\epsilon}$, $1\leq p < \infty$, associated with the phase-field function $\omega_{\epsilon}$ as follows:
\begin{equation*}
	L^p(D_{\epsilon};\omega_{\epsilon})=\left\{v\,\Big|\,\int_{D_{\epsilon}} |v|^p\,\omega_{\epsilon}\dd \bm x<\infty \right\},
\end{equation*}
with the norm
\begin{align*}
	\|v\|_{L^p(D_{\epsilon};\omega_{\epsilon})}=\left(\int_{D_{\epsilon}}|v|^p \,\omega_{\epsilon}\dd \omega_{\epsilon}\right)^{\frac 1 p}.
\end{align*}
Based on the weighted spaces $L^p(D_{\epsilon};\omega_{\epsilon})$,  the weighted Sobolev spaces are consequently defined as
\begin{equation*}
	W^{s,p}(D_{\epsilon};\omega_{\epsilon})=\left\{v\in L^p(D_{\epsilon};\omega_{\epsilon})\,|\, D^{\bm \alpha}v \in L^p(D_{\epsilon};\omega_{\epsilon}), \ \forall\, |\bm\alpha|\leq s \right\},
\end{equation*}
and $H^s(D_{\epsilon};\omega^{\epsilon}):=W^{s,2}(D_{\epsilon};\omega_{\epsilon})$ with the norm 
\begin{equation*}
	\|v\|_{W^{s,p}(D_{\epsilon};\omega_{\epsilon})}=\left(\int_{D_{\epsilon}}\sum\limits_{|\bm\alpha|\leq s}\left|D^{\bm \alpha}v\right|^p\dd\omega_{\epsilon}\right)^{\frac 1 p}.
\end{equation*}

Therefore, \eqref{approx1} and \eqref{approx2} leads to the following  diffuse domain approximation of \eqref{eq2-2} in $D_{\epsilon}$: find  $u^{\epsilon} \in L^2(0,T;H^1(D_{\epsilon};\omega^{\epsilon}))$ such that
\begin{equation}
	\label{weighted_variational}
	\left(u_t^{\epsilon},v \right)_{D_{\epsilon};\omega^{\epsilon}}+a^{\epsilon}(u^{\epsilon},v)= \ell^{\epsilon}(v), \qquad \forall \,v \in H^1(D_{\epsilon};\omega^{\epsilon}),
\end{equation}
where
\begin{equation*}
	\begin{split}
		a^{\epsilon}(w,v) &= \int_{D_{\epsilon}} \widetilde A\nabla w\cdot\nabla v \,\omega^{\epsilon}\dd \bm x,\qquad \forall\,w, v\in H^1(D_{\epsilon};\omega^{\epsilon}),\\
		\ell^{\epsilon}(v) &= \int_{D_{\epsilon}} \widetilde f v\omega^{\epsilon}\dd \bm x + \int_{D_{\epsilon}} \widetilde gv\left|\nabla \omega^{\epsilon} \right|\dd \bm x,\qquad \forall\, v\in H^1(D_{\epsilon};\omega^{\epsilon})
	\end{split}
\end{equation*}
with the initial value $u^{\epsilon}(0) = \widetilde{u_0}$. Here $\widetilde A$, $\widetilde{u_0}$, $\widetilde f(t)$  and $\widetilde g(t)$ are certain extensions of $A$, $u_0$, $f(t)$  and $g(t)$ from $D$ to $D_{\epsilon}$, respectively. The extension details will be illustrated in the following.

Let us  also assume the domain boundary $\partial D$ is of class $C^{1,1}$ from now on in this paper.  Define the $\epsilon$-tubular neighborhood of $\partial D$ by
\begin{equation*}
	\Gamma_{\epsilon} = D_{\epsilon}\setminus \overline{D_{-\epsilon}}.
\end{equation*}
It holds that $\dist(\bm x,\partial D)< \epsilon$ for all $\bm x \in \Gamma_{\epsilon}$ and $\dist(\bm x,\partial D)\geq \epsilon$ for all $\bm x \in \Omega\setminus \Gamma_{\epsilon}$. 
Due to the $C^{1,1}$ regularity of $\partial D$, the projection of $\bm z \in \Gamma_{\epsilon}$ onto $\partial D$ is unique for  sufficiently small $\epsilon$ \cite{Adams1975}, i.e., for each $z \in \Gamma_{\epsilon}$, there exists a unique $\bm x \in \partial D$ such that $\bm z = \bm x+d_{D}(\bm z)\bm n(\bm x)$, where $\bm n(\bm x)$ is the outward unit normal vector for $\bm x\in \partial D$. 	
In order to keep the regularity of the source function $f(t)$ and the initial value $u_0$, one can, for example,  extend them by the reflection with respect to the outward normal direction of  the domain  boundary $\partial D$. For $L^p$-functions, this can be achieved by an extension by constant and for the $W^{k,p}$-functions, this can be achieved by the construction shown in the Chapter 4 of \cite{Adams1975}.
Regarding the extension of Neumann boundary value, since $g(t) \in L^2(\partial D)$, $g(t)$ is defined a.e. on $\partial D$ and we can define an extension of $g(t)$ a.e. on $\Gamma_{\epsilon}$ by
\begin{equation*}
	\widetilde{g}(t,\bm x+s\bm n(\bm x))=g(t,\bm x), \quad -\epsilon\leq s\leq \epsilon, \ \bm x \in \partial D,
\end{equation*}
and further simply take $\widetilde{g}(t,\bm x)=0$ for $\bm x\in D_{\epsilon}\setminus \Gamma_{\epsilon}$.

Based on the definition of weight function $\omega_{\epsilon}$, it's obvious that $|\nabla \omega_{\epsilon}(\bm x)|\neq 0$ only for $\bm x \in \Gamma_{\epsilon}$.	
Since the weighted function  $\omega_{\epsilon}(\bm x)$ and its gradient $\nabla \omega_{\epsilon}(\bm x)$ will vanish on $\partial D_{\epsilon}$ and outside of  $D_{\epsilon}$,  we can further extend the integration on $D_{\epsilon}$ to the larger rectangular domain $\Omega$, which is easy for generating spatial grids in practice. The weighted Sobolev space and corresponding variational problem \eqref{eq2-2} defined in $D_{\epsilon}$ is then  transformed to 
a similar problem defined in $\Omega$: find  $u^{\Omega} \in L^2(0,T;H^1(\Omega;\omega^{\epsilon}))$ such that
\begin{equation}
	\label{weighted_variational_1}
	\left(u_t^{\Omega},v\right)_{\Omega;\omega^{\epsilon}}+a^{\Omega}(u^{\Omega},v)= \ell^{\Omega}(v), \quad \forall\, v \in H^1(\Omega;\omega^{\epsilon}), 
\end{equation}
where
\begin{equation*}
	\begin{split}
		a^{\Omega}(w,v) &= \int_{\Omega} \overline{A}\nabla w\cdot\nabla v \,\omega^{\epsilon}\dd \bm x,\qquad \forall\,w, v\in H^1(\Omega;\omega^{\epsilon}),\\
		\ell^{\Omega}(v) &= \int_{\Omega} \overline{f}v\,\omega^{\epsilon}\dd \bm x + \int_{\Omega} \overline gv\left|\nabla \omega^{\epsilon} \right|\dd \bm x,\qquad \forall v\in H^1(\Omega;\omega^{\epsilon})
	\end{split}
\end{equation*}
with the initial value $u^{\Omega}(0)=\overline{u_0}$. Here $\overline{A}$, $\overline{u_0}$, $\overline f(t)$ and $\overline g(t)$ are  the extensions of $\widetilde{A}$, $\widetilde{u_0}(\bm x)$, $\widetilde f(t)$ and $\widetilde g(t)$ from $D_{\epsilon}$ to $\Omega$ and the choices of extensions are not unique. Since 
$\omega_{\epsilon}(\bm x)=0$ for $\bm x \in \Omega\setminus D_{\epsilon}$ and consequently $u^{\epsilon}|_{D_{\epsilon}}=u^{\Omega}|_{D_{\epsilon}}$, we only need investigate convergence of the diffuse domain solution $u^\epsilon$ defined on $D_{\epsilon}$ produced from the variational  problem \eqref{weighted_variational} to the exact solution $u$ of the original variational problem \eqref{eq2-2}.

\section{Preliminaries}\label{pre}

Note that $\Gamma_{\epsilon}$ can be rewritten as
\begin{equation*}
	\Gamma_{\epsilon}=\left\{\bm z \in \Omega\,\big|\, \exists\  \bm x \in \partial D,\  |s|<\epsilon,\ \bm z =\bm x + s\bm n(\bm x) \right\}.
\end{equation*}
Furthermore, 
$\left|\Gamma_{\epsilon} \right|\lesssim \epsilon\mathcal{H}^{n-1}(\partial D)$,
where $\left|\Gamma_{\epsilon}\right|=\mathcal{L}^n(\Gamma_{\epsilon})$ is the $n$-dimensional Lebesgue measure of $\Gamma_{\epsilon}$ and $\mathcal{H}^{n-1}(\partial D)$ is the $(n-1)$-dimensional Hausdorff measure of $\partial D$. 

\subsection{On the weighted Sobolev space}

The following three theorems (Theorems \ref{trace_thm}, \ref{embed_thm} and \ref{Poincare_thm})
have readily been proved in \cite{BurgerElvetun2017}.
\begin{theorem}[Trace Theorem]
	\label{trace_thm}
	Let $\epsilon_0>0$ be sufficiently small and $1\leq p < \infty$. Then, there exists a constant $C>0$ such that for any $\epsilon \in [0, \epsilon_0]$ and $v \in W^{1,p}(D_{\epsilon};\omega_{\epsilon})$, there holds
	\begin{equation*}
		\int_{D_{\epsilon}} |v|^p|\nabla \omega_{\epsilon}|\dd \bm x \lesssim \|v\|_{W^{1,p}(D_{\epsilon};\omega_{\epsilon})}^p.
	\end{equation*}
\end{theorem}

\begin{theorem}[Embedding Theorem]\label{embed_thm}
	Suppose that $\epsilon \in (0, \epsilon_0]$, $1\leq p < \infty$, and $\alpha>0$ be the constant satisfies that for all $s \in (0, 2)$, $\zeta_1 s^{\alpha}\leq (1+S(s-1))/2\leq \zeta_2 s^{\alpha}$ for some constants $\zeta_1,\zeta_2>0$. Then the following embeddings are continuous
	\begin{equation*}
		W^{1,p}(D_{\epsilon};\omega_{\epsilon}) \hookrightarrow L^q(D_{\epsilon};\omega_{\epsilon}), \quad 1 \leq q \leq p_{\alpha}^*, \ q < \infty,
	\end{equation*}
	where
	$	p_{\alpha}^*=\frac{p(n+\alpha)}{n+\alpha-p}$ for $p < n+\alpha$ and  $p_{\alpha}^*=\infty$ for $p \geq n+\alpha.$
	Moreover, there exists a constant $C$ independent of $\epsilon$ such that for any $v \in W^{1,p}(D_{\epsilon};\omega_{\epsilon})$, there holds
	\begin{equation*}
		\|v\|_{L^q}(D_{\epsilon};\omega_{\epsilon}) \lesssim \|v\|_{W^{1,p}(D_{\epsilon};\omega_{\epsilon})}.
	\end{equation*}
\end{theorem}

\begin{theorem}[Poincare-Friedrichs-type inequality]\label{Poincare_thm}
	Suppose that $\epsilon \in (0,\epsilon_0]$, $1\leq p<\infty$, and  the extended domain $D_{\epsilon}$ be connected. Then, there exists a constant $C$ independent of $\epsilon$ such that for any $\epsilon \in (0,\epsilon_0)$ and $v \in W^{1,p}(D_{\epsilon};\omega_{\epsilon})$, there holds
	\begin{equation*}
		\|v\|_{L^p(D_{\epsilon};\omega_{\epsilon})} \lesssim \|\nabla v\|_{L^p(D_{\epsilon};\omega_{\epsilon})}^p + \int_{D_{\epsilon}}|v|^p|\nabla \omega_{\epsilon}|\dd \bm x.
	\end{equation*}
	
\end{theorem}

\subsection{On diffuse volume integrals}

The following two theorems readily come from Theorem 5.2 and Theorem 5.6 in \cite{BurgerElvetun2017}.

\begin{theorem}
	\label{thm1}
	Suppose that $\epsilon\in (0, \epsilon_0]$ and  $h(\bm x) \in H^1(D_{\epsilon};\omega_{\epsilon})$. Then, 
	\begin{equation*}
		\left|\int_{D_{\epsilon}} h(\bm x)\dd \omega_{\epsilon}(\bm x)-\int_D h(\bm x)\dd \bm x \right| \lesssim \epsilon^{\frac 3 2}\|h\|_{H^1(\Gamma_{\epsilon};\omega_{\epsilon})},
	\end{equation*}
	where the hidden constant is independent of $\epsilon$.
\end{theorem}

\begin{theorem}
	\label{thm2}
	Suppose that  $\epsilon\in (0, \epsilon_0]$, $w \in H^2(D_{\epsilon};\omega_{\epsilon})$ satisfying $w=0$ on $\partial D$ and  $v \in H^1(D_{\epsilon};\omega_{\epsilon})$. Then, there holds
	\begin{equation*}
		\int_{\Gamma_{\epsilon}} wv\left|\nabla \omega_{\epsilon}\right|\dd \bm x \lesssim \left(\epsilon^{\frac 3 2}\|w\|_{H^2(D_{\epsilon};\omega_{\epsilon})}+\epsilon^{\frac 3 2}\|w\|_{H^2(\Gamma_{\epsilon};\omega_{\epsilon})}\right)\|v\|_{H^1(D_{\epsilon};\omega_{\epsilon})},
	\end{equation*}
	where the hidden constant is independent of $\epsilon$, $u$ and $v$.
\end{theorem}

Next, let us introduce the smooth condition for  $f$ and some regularity condition required for the exact solution $u$ in order to carry out the convergence and error analysis of the diffuse domain method. For simplicity of expressions, in the following analysis we will not distinguish the exact solution $u(t)$, the source function $f(t)$, the Neumann  boundary value $g(t)$, the diffusion coefficient $A$ with their extensions  in $D_\epsilon$ since there are no ambiguities.

\begin{assumption}
	\label{assumption1}
	The extension of the source function $f(t,\bm x)$ to $D_{\epsilon}$ satisfies the following regularity condition:
	\begin{subequations}
		\begin{align}
			\label{smooth_cond1}
			\sup\limits_{0\leq t \leq T}\|f(t,\bm x)\|_{L^2(D_{\epsilon};\omega_{\epsilon})} &\lesssim 1,\\
			\label{smooth_cond2}
			\sup\limits_{0\leq t \leq T}\|f(t,\bm x)\|_{H^1(\Gamma_{\epsilon};\omega_{\epsilon})}&\lesssim 1,
		\end{align}
	\end{subequations}
	where the hidden constant may depend on the terminal time $T$.
\end{assumption}
\begin{assumption}
	\label{regularity_assump}
	The extension of the exact solution $u(t)$ of \eqref{eq2-2}   to $D_\epsilon$ satisfies the following regularity conditions:
	\begin{subequations}
		\begin{align}
			\label{regularity1}
			\sup\limits_{0\leq t \leq T}\|u(t)\|_{H^2(D_{\epsilon};\omega_{\epsilon})} &\lesssim 1,\\
			\label{regularity2}
			\sup\limits_{0\leq t \leq T}\|u_t(t)\|_{W^{1,\infty}(D_{\epsilon})}&\lesssim 1,
		\end{align}
	\end{subequations}
	where the hidden constants may depend on the terminal time  $T$.
\end{assumption}

\begin{lemma}
	\label{lemma1}
	Suppose that  $0<\epsilon\leq \epsilon_0$, $w\in H^1(D_{\epsilon};\omega_{\epsilon})\cap L^{\infty}(D_{\epsilon})$, and $v\in H^1(D_{\epsilon};\omega_{\epsilon})$. Then
	\begin{equation*}
		\left|\int_{D_{\epsilon}}w v\omega_{\epsilon}\dd \bm x-\int_D w v\dd \bm x \right|\lesssim \epsilon^{\frac 1 2} \|w\|_{L^2(\Gamma_{\epsilon};\omega_{\epsilon})}\|v\|_{L^2(\Gamma_{\epsilon};\omega_{\epsilon})},
	\end{equation*}
	where the hidden constant  is independent of $\epsilon$, $u$ and $v$.
\end{lemma}
\begin{proof}
	According to the definition of $D_{\epsilon}$ and $\omega_{\epsilon}$, we can rewrite 
	\begin{equation}\label{lemma1-1}
		\begin{aligned}
			& \int_{D_{\epsilon}}w v\omega_{\epsilon} \dd \bm x-\int_D w v\dd \bm x\\
			&\quad =\int_{D_{-\epsilon}}wv\omega_{\epsilon} \dd \bm x+\int_{D_{\epsilon}\setminus D_{-\epsilon}}wv\omega_{\epsilon}\dd \bm x-\int_{D_{-\epsilon}}wv\dd \bm x-\int_{D\setminus D_{-\epsilon}} wv\dd \bm x\\
			&\quad =\int_{D\setminus D_{-\epsilon}}wv\omega_{\epsilon}\dd \bm x+\int_{D_{\epsilon}\setminus D}wv\omega_{\epsilon}\dd \bm x-\int_{D\setminus D_{-\epsilon}} wv\dd \bm x\\
			&\quad =\int_{D\setminus D_{-\epsilon}} \left(w-w/\omega_{\epsilon}\right)v\omega_{\epsilon}\dd \bm x+\int_{D_{\epsilon}\setminus D}wv\omega_{\epsilon}\dd \bm x.
		\end{aligned}
	\end{equation}
	By applying  the Cauchy-Schwarz inequality and the triangle inequality, \eqref{lemma1-1} can be bounded by
	\begin{equation}\label{lemma1-11}
		\begin{aligned}
			&\int_{D_{\epsilon}}w v\omega_{\epsilon} \dd \bm x-\int_D w v\dd \bm x\\
			&\quad \leq \left(\int_{D\setminus D_{-\epsilon}}\left|w-w/\omega_{\epsilon} \right|^2\omega_{\epsilon}\dd \bm x \right)^{\frac 1 2}\left(\int_{D\setminus D_{-\epsilon}}|v|^2\omega_{\epsilon}\dd \bm x\right)^{\frac 1 2}\\
			&\quad\quad +\left(\int_{D_{\epsilon}\setminus D}|w|^2\omega_{\epsilon}\dd \bm x\right)^{\frac 1 2}\left(\int_{D_{\epsilon}\setminus D} |v|^2\omega_{\epsilon}\dd \bm x\right)^{\frac 1 2}\\
			&\quad \lesssim  \left(\int_{\Gamma_{\epsilon}} |v|^2\omega_{\epsilon} \dd \bm x\right)^{\frac 1 2}\left(\int_{D\setminus D_{-\epsilon}}\left|w-w/\omega_{\epsilon} \right|^2\omega_{\epsilon}\dd \bm x+\int_{D_{\epsilon}\setminus D} |w|^2\omega_{\epsilon}\dd\bm x \right)^{\frac 1 2}\\
			&\quad \lesssim \left(\int_{\Gamma_{\epsilon}} |v|^2\omega_{\epsilon} \dd \bm x\right)^{\frac 1 2} \Big({\rm{I}}_1+{\rm{I}}_2 \Big)^{\frac 1 2}
		\end{aligned}
	\end{equation}
	where
	$${\rm{I}}_1:=\int_{D_{\epsilon}\setminus D_{-\epsilon}} |w|^2\omega_{\epsilon}\dd \bm x - \int_{D\setminus D_{-\epsilon}}|w|^2\dd \bm x,$$
	$${\rm{I}}_2:=\int_{D\setminus D_{-\epsilon}}|w|^2/\omega_{\epsilon}\dd \bm x-\int_{D\setminus D_{-\epsilon}}|w|^2\dd \bm x. $$
	
	As for the estimate of $\rm{I}_1$, with the help of Theorem \ref{thm1}, we can derive that
	\begin{align}
		\label{lemma1-2}
		|{\rm{I}_1}| &= \left|\int_{D_{\epsilon}}|w|^2\omega_{\epsilon}\dd \bm x-\int_D|w|^2\dd \bm x\right|\lesssim \epsilon^{\frac 3 2}\|w^2\|_{H^1(\Gamma_{\epsilon};\omega_{\epsilon})}.
	\end{align}
	
	As for the estimate of $\rm{I}_2$, recalling the definition of $D_{\epsilon}$ and $\omega_{\epsilon}$, we can derive that
	{\footnotesize\begin{equation}\label{lemma1-3}
			\begin{aligned}
				|{\rm{I}_2}|=&\left|\int_{-\epsilon}^0 \frac{\frac{6}{\epsilon}S'\left(-\frac{3s}{\epsilon}\right)}{\left(1+S\left(-\frac{3s}{\epsilon}\right)\right)^2}\int_{\{d_D(\bm x)<s\}}|w|^2\dd\bm x\dd s -\int_{-\epsilon}^0\frac{\frac {6}{\epsilon}S'\left(-\frac{3s}{\epsilon}\right)}{\left(1+S\left(-\frac{3s}{\epsilon}\right)\right)^2}\int_{\{d_D(\bm x)<0\}}|w|^2\dd\bm x\dd s  \right|\\
				=& \left|\int_{-\epsilon}^0 \frac{\frac{6}{\epsilon}S'\left(-\frac{3s}{\epsilon}\right)}{\left(1+S\left(-\frac{3s}{\epsilon}\right)\right)^2}\int_{\{s<d_D(\bm x)<0\}}|w|^2\dd\bm x\dd s\right|\\
				\lesssim&\, \frac{1}{\epsilon}\left(\int_{-\epsilon}^0 S'\left(-\frac{3s}{\epsilon}\right)\int_{\{s<d_D(\bm x)<0\}}1\dd \bm x \dd s\right)^{\frac 1 2}\left(\int_{-\epsilon}^0 S'\left(-\frac{3s}{\epsilon}\right)\int_{\{s<d_D(\bm x)<0\}}|w|^4\dd \bm x\dd s \right)^{\frac 1 2}\\
				\lesssim&\, \epsilon^{\frac 1 2}\left(\int_{-\epsilon}^0 \frac{1}{\epsilon} S'\left(-\frac{3s}{\epsilon}\right)\int_{\{s<d_D(\bm x)<0 \}}|w|^4\dd\bm x\dd s \right)^{\frac 1 2}\\
				\lesssim&\, \epsilon^{\frac 1 2}\left(\int_0^1\int_{\{-s<\varphi^{\epsilon}<0 \}}|w|^4\dd \bm x\dd s \right)^{\frac 1 2},
			\end{aligned}
	\end{equation}}
	where the last three inequalities use the Cauchy-Schwarz inequality and the boundness of $S'(\cdot)$. Since $w \in L^{\infty}(D_{\epsilon})$, \eqref{lemma1-3} can be reformulated as
	\begin{align}
		\label{lemma1-4}
		{\rm{I}_2} &\lesssim \epsilon^{\frac 1 2}\|w\|_{L^2(\Gamma_{\epsilon};\omega_{\epsilon})}.
	\end{align}
	
	Combining \eqref{lemma1-11}, \eqref{lemma1-2}, \eqref{lemma1-4}, we finally obtain
	\begin{equation*}
		\left|\int_{D_{\epsilon}} wv\omega_{\epsilon}\dd \bm x-\int_D wv\dd \bm x \right|\lesssim \epsilon^{\frac 1 2}\|w\|_{L^2(\Gamma_{\epsilon};\omega_{\epsilon})}\|v\|_{L^2(\Gamma_{\epsilon};\omega_{\epsilon})},
	\end{equation*}
	which complete the proof.	
	
\end{proof}

\section{Error analysis}\label{theory}
To illustrate the convergence of diffuse domain method, let us combine the \eqref{eq2-2} and \eqref{weighted_variational} together, then we can get the error equation as follows: for all $v \in H^1(D_{\epsilon};\omega_{\epsilon})$,
\begin{equation}
	\label{eq4-1}
	\begin{split}
		\left(u_t^{\epsilon}-u_t,v\right)_{D_{\epsilon};\omega_{\epsilon}}+a^{\epsilon}(u^{\epsilon}-u,v)
		=&[\left(u_t,v\right)-\left(u_t,v\right)_{D_{\epsilon};\omega_{\epsilon}}]+[a(u,v)-a^{\epsilon}(u,v)]\\
		&+[\ell^{\epsilon}(v)-\ell(v)].
	\end{split}
\end{equation}
Next we will derive the error estimates of $u^{\epsilon}-u$ in the $L^2$ and $H^1$ norms.

\subsection{Optimal error estimate in the $L^2$-norm}

\begin{theorem}[Error estimate in the $L^2$ norm]
	\label{thm3}
	Suppose that $0<\epsilon\leq \epsilon_0$,  $g(t)\in H^2(D_{\epsilon};\omega_{\epsilon})$ for $t\in (0,T]$, $\kappa\leq A(\bm x)\leq \kappa^{-1}$ for all $\bm x \in D_{\epsilon}$ with some constant $\kappa>0$, and $f$ satisfies  Assumption \ref{assumption1}. Assume the exact solution $u\in L^2(0,T;H^2(D))$ and its extension  to $D_\epsilon$ fulfills Assumption  \ref{regularity_assump}. Then we have
	\begin{equation}
		\label{thm3-1}
		\begin{split}
			\|u^{\epsilon}(t)-u(t)\|_{L^2(D_{\epsilon};\omega_{\epsilon})}\lesssim \epsilon^2, \quad \forall \,0\leq t \leq T,
		\end{split}
	\end{equation}
	where the hidden constant is independent of $\epsilon$.
\end{theorem}
\begin{proof}
	Let us analyze $\left(u_t,v\right)-\left(u_t,v\right)_{D_{\epsilon};\omega_{\epsilon}}$, $a(u,v)-a^{\epsilon}(u,v)$ and $\ell^{\epsilon}(v)-\ell(v)$ respectively for any $v \in H^1(D_{\epsilon};\omega_{\epsilon})$. 
	Recalling the definition of $a^{\epsilon}(\cdot,\cdot)$ and $a(\cdot,\cdot)$, since $\omega_{\epsilon}$ will vanish on  $\partial D_{\epsilon}$, we can derive the following formula by using the integration by part,
	{\small\begin{equation}\label{thm3-2}
			\begin{aligned}
				&a(u,v)-a^{\epsilon}(u,v)\\
				&\quad=\int_{D_{\epsilon}} \divo(A\nabla u)\omega_{\epsilon}v\dd\bm x+\int_{D_{\epsilon}} A\nabla u\cdot\nabla \omega_{\epsilon} v\dd \bm x+\int_{\partial D} \bm n\cdot\left(A\nabla u\right)v\dd \sigma-\int_D\divo(A\nabla u)v\dd \bm x\\
				&\quad=\int_{D_{\epsilon}}\divo\left(A\nabla u\right)\omega_{\epsilon} v\dd\bm x-\int_D\divo\left(A\nabla u\right)v\dd \bm x-\int_{D_{\epsilon}} \bm n\cdot(A\nabla u)\left|\nabla \omega_{\epsilon}\right|v\dd \bm x+\int_{\partial D} gv\dd \sigma,
			\end{aligned}
	\end{equation}}
	where in the last equality we use the fact that $\nabla \omega_{\epsilon}=-\bm n\left|\nabla \omega_{\epsilon}\right|$. By Theorem \ref{thm1}, we get
	\begin{align}
		\label{thm3-3}
		&\left|\int_D \divo(A\nabla u)v\dd \bm x-\int_{D_{\epsilon}} \divo(A\nabla u)v\omega_{\epsilon}\dd \bm x\right|
		\lesssim \epsilon^{\frac 3 2}\left\|\divo(A\nabla u)\right\|_{H^1(\Gamma_{\epsilon};\omega_{\epsilon})}\|v\|_{H^1(\Gamma_{\epsilon};\omega_{\epsilon})}.
	\end{align}
	Then, inserting \eqref{thm3-3} to \eqref{thm3-2}, we get that
	\begin{equation}
		\begin{aligned}
			\label{thm3-4}
			a(u,v)-a^{\epsilon}(u,v)\lesssim& \epsilon^{\frac 3 2}\left\|\divo(A\nabla u)\right\|_{H^1(\Gamma_{\epsilon};\omega_{\epsilon})}\|v\|_{H^1(\Gamma_{\epsilon};\omega_{\epsilon})}-\int_{D_{\epsilon}} \bm n\cdot(A\nabla u)\left|\nabla \omega_{\epsilon}\right|v\dd \bm x\\
			&+\int_{\partial D} gv\dd \sigma.
		\end{aligned}
	\end{equation}
	
	As for the estimate of $\left(u_t,v\right)-\left(u_t,v\right)_{D_{\epsilon};\omega_{\epsilon}}$, by using the conclusion of Theorem \ref{thm1}, we obtain
	\begin{equation}
		\label{thm3-5}
		\left(u_t,v\right)-\left(u_t,v\right)_{D_{\epsilon};\omega_{\epsilon}}\lesssim \epsilon^{\frac 3 2}\|u_t\|_{H^1(\Gamma_{\epsilon};\omega_{\epsilon})}\|v\|_{H^1(\Gamma_{\epsilon};\omega_{\epsilon})}.
	\end{equation}
	As for the estimate of $\ell^{\epsilon}(v)-\ell(v)$, same as the derivation of \eqref{thm3-5}, we can get
	\begin{equation}
		\label{thm3-6}
		\begin{aligned}
			\ell^{\epsilon}(v)-\ell(v)&= \int_{D_{\epsilon}}fv\omega_{\epsilon}\dd \bm x-\int_D fv\dd \bm x+\int_{D_{\epsilon}}gv\left|\nabla\omega_{\epsilon}\right|\dd \bm x-\int_{\partial D}gv\dd \sigma\\
			&\lesssim \epsilon^{\frac 3 2}\|f\|_{H^1(\Gamma_{\epsilon};\omega_{\epsilon})}\|v\|_{H^1(\Gamma_{\epsilon};\omega_{\epsilon})}+\int_{D_{\epsilon}} gv\left|\nabla\omega_{\epsilon}\right|\dd \bm x-\int_{\partial D}gv\dd \sigma.
		\end{aligned}
	\end{equation}
	Combining \eqref{thm3-4}-\eqref{thm3-6} together, we  get
	\begin{equation}
		\label{thm3-7}
		\begin{aligned}
			&\left(u_t^{\epsilon}-u_t,v\right)_{D_{\epsilon};\omega_{\epsilon}}+a^{\epsilon}(u^{\epsilon}-u,v)\\
			&\quad=\left(u_t,v\right)-\left(u_t,v\right)_{D_{\epsilon};\omega_{\epsilon}}+a(u,v)-a^{\epsilon}(u,v)+\ell^{\epsilon}(v)-\ell(v)\\	
			&\quad\lesssim\epsilon^{\frac 3 2}\|u_t\|_{H^1(\Gamma_{\epsilon};\omega_{\epsilon})}\|v\|_{H^1(\Gamma_{\epsilon};\omega_{\epsilon})}+\epsilon^{\frac 3 2}\left\|\divo\left(A\nabla u\right)\right\|_{H^1(\Gamma_{\epsilon};\omega_{\epsilon})}\|v\|_{H^1(\Gamma_{\epsilon};\omega_{\epsilon})}\\
			&\qquad-\int_{D_{\epsilon}} \bm n\cdot(A\nabla u)\left|\nabla\omega_{\epsilon}\right|v\dd\bm x+\epsilon^{\frac 3 2}\|f\|_{H^1(\Gamma_{\epsilon};\omega_{\epsilon})}\|v\|_{H^1(\Gamma_{\epsilon};\omega_{\epsilon})}+\int_{D_{\epsilon}} gv|\nabla\omega_{\epsilon}|\dd \bm x.
		\end{aligned}
	\end{equation}
	Applying with Theorem \ref{thm2}, we know that
	\begin{equation}\label{thm3-8}
		\begin{aligned}
			\int_{D_{\epsilon}} gv|\nabla\omega_{\epsilon}|\dd \bm x-\int_{D_{\epsilon}}\bm n\cdot(A\nabla u)v|\nabla\omega_{\epsilon}|\dd \bm x
			&=\int_{D_{\epsilon}}\left(\bm n\cdot(A\nabla u)-g \right)v|\nabla \omega_{\epsilon}|\dd \bm x\notag\\
			&\lesssim\epsilon^{\frac 3 2}\|\bm n\cdot(A\nabla u)-g\|_{H^2(\Gamma_{\epsilon};\omega_{\epsilon})}\|v\|_{H^1(D_{\epsilon};\omega_{\epsilon})}.
		\end{aligned}
	\end{equation}
	Therefore, we have by  \eqref{thm3-7}   and \eqref{thm3-8} that
	\begin{equation}\label{thm3-9}
		\begin{aligned}
			&\left(u_t^{\epsilon}-u_t,v\right)_{D_{\epsilon};\omega_{\epsilon}}+a^{\epsilon}(u^{\epsilon}-u,v)\\
			&\quad\lesssim\epsilon^{\frac 3 2}\|u_t\|_{H^1(\Gamma_{\epsilon};\omega_{\epsilon})}\|v\|_{H^1(\Gamma_{\epsilon};\omega_{\epsilon})}+\epsilon^{\frac 3 2}\left\|\divo\left(A\nabla u\right)\right\|_{H^1(\Gamma_{\epsilon};\omega_{\epsilon})}\|v\|_{H^1(\Gamma_{\epsilon};\omega_{\epsilon})}\\
			&\qquad+\epsilon^{\frac 3 2}\|\bm n\cdot(A\nabla u)-g\|_{H^2(\Gamma_{\epsilon};\omega_{\epsilon})}\|v\|_{H^1(D_{\epsilon};\omega_{\epsilon})}+\epsilon^{\frac 3 2}\|f\|_{H^1(\Gamma_{\epsilon};\omega_{\epsilon})}\|v\|_{H^1(D_{\epsilon};\omega_{\epsilon})}.
		\end{aligned}
	\end{equation}
	
	Since $|\Gamma_{\epsilon}|\lesssim \epsilon$, applying $v=u^{\epsilon}-u$ to \eqref{thm3-9} gives
	\begin{equation}
		\begin{aligned}
			\label{thm3-10}
			&\frac 1 2\frac{\dd}{\dd t}\|u^{\epsilon}-u\|_{L^2(D_{\epsilon};\omega_{\epsilon})}^2+\|\nabla(u^{\epsilon}-u)\|_{L^2(D_{\epsilon};\omega_{\epsilon})}^2\\
			&\quad\lesssim \epsilon^{\frac 3 2}\Big(\|u_t\|_{H^1(\Gamma_{\epsilon};\omega_{\epsilon})}+\|\divo(A\nabla u)\|_{H^1(\Gamma_{\epsilon};\omega_{\epsilon})}+\|\bm n\cdot(A\nabla u)-g\|_{H^2(\Gamma_{\epsilon};\omega_{\epsilon})}\notag\\
			&\qquad+\|f\|_{H^1(\Gamma_{\epsilon};\omega_{\epsilon})}\Big)\|u^{\epsilon}-u\|_{H^1(\Gamma_{\epsilon};\omega_{\epsilon})}\notag\\
			&\quad\lesssim  \epsilon^4 +\|u^{\epsilon}-u\|_{L^2(D_{\epsilon};\omega_{\epsilon})}^2+\|\nabla(u^{\epsilon}-u)\|_{L^2(D_{\epsilon};\omega_{\epsilon})}^2.
		\end{aligned}
	\end{equation}
	Thus
	\begin{equation*}
		\frac 1 2\frac{\dd}{\dd t}\|u^{\epsilon}-u\|_{L^2(D_{\epsilon};\omega_{\epsilon})}^2 \lesssim \epsilon^4+\|u^{\epsilon}-u\|_{L^2(D_{\epsilon};\omega_{\epsilon})}^2,
	\end{equation*}
	which implies 
	\begin{equation*}
		\|u^{\epsilon}-u\|_{L^2(D_{\epsilon};\omega_{\epsilon})}\lesssim \epsilon^2.
	\end{equation*}
	The proof is completed.
	
\end{proof}

\subsection{Optimal error estimate in the $H^1$ norm}

First we derive several important lemmas needed for the $H^1$ norm error estimate.

\begin{lemma}
	\label{lemma3}
	Suppose that $0<\epsilon\leq \epsilon_0$,  $g(t)\in H^1(D_{\epsilon};\omega_{\epsilon}) $, and $f$ satisfies  Assumption \ref{assumption1}. Assume the exact solution $u\in L^2(0,T;H^1(D))$ and its extension  to $D_\epsilon$ fulfills Assumption  \ref{regularity_assump}. Then we have
	\begin{equation}
		\label{lemma3-1}
		\int_{\partial D}g(u^{\epsilon}-u)\dd \sigma+\int_{D_{\epsilon}} g|\nabla\omega_{\epsilon}|(u^{\epsilon}-u)\dd \bm x\leq C\epsilon^2+\frac 1 4\|\nabla u^{\epsilon}-\nabla u\|_{L^2(\Gamma_{\epsilon};\omega_{\epsilon})}^2,
	\end{equation}
	where $C$ is a constant independent of $\epsilon$.
\end{lemma}
\begin{proof}
	Using $\nabla d_D(\bm x)=\bm n(\bm x)$, $\nabla \omega_{\epsilon}=-\bm n|\nabla\omega_{\epsilon}|$ and the divergence theorem, we derive that
	\begin{equation}
		\label{lemma3-2}
		\begin{aligned}
			&\int_{\partial D}g(u^{\epsilon}-u)\dd \sigma+ \int_{D_{\epsilon}}g|\nabla\omega_{\epsilon}|(u^{\epsilon}-u)\dd \bm x\\
			&\quad=\int_{\partial D}g(u^{\epsilon}-u)\dd \sigma-\int_{\partial D_{\epsilon}}g|\nabla \omega_{\epsilon}|(u^{\epsilon}-u)\dd \bm x\\
			&\quad=\int_D \divo(g\nabla d_D(u^{\epsilon}-u))\dd \bm x-\int_{D_{\epsilon}}\divo(g\nabla d_D(u^{\epsilon}-u))\omega_{\epsilon}\dd \bm x\\
			&\quad=\left|\int_D \divo(g\nabla d_D)\left(u^{\epsilon}-u\right)\dd \bm x-\int_{D_{\epsilon}}\divo(g\nabla d_D)\left(u^{\epsilon}-u\right)\omega_{\epsilon} \dd \bm x\right|\\
			&\qquad+\left|\int_D g\nabla d_D\cdot\left(\nabla u^{\epsilon}-\nabla u\right)\dd \bm x+\int_{D_{\epsilon}}g\nabla d_D\cdot\left(\nabla u^{\epsilon}-\nabla u\right)|\omega_{\epsilon}|\dd \bm x\right| \\
			&\quad=:\rm{II}_1+\rm{II}_2.
		\end{aligned}
	\end{equation}
	By using the conclusion of Lemma \ref{lemma1} and $|\Gamma_{\epsilon}|\lesssim \epsilon$, we obtain
	\begin{equation}
		\label{lemma3-3}\left\{
		\begin{aligned}
			\rm{II}_1&\lesssim \left\|\divo(g\nabla d_D)\right\|_{L^2(\Gamma_{\epsilon};\omega_{\epsilon})}\epsilon^{\frac 1 2}\|u^{\epsilon}-u\|_{L^2(\Gamma_{\epsilon};\omega_{\epsilon})}\leq C\epsilon^2+\frac 1 8\|u^{\epsilon}-u\|_{L^2(\Gamma_{\epsilon};\omega_{\epsilon})}^2,\\
			\rm{II}_2 &\lesssim \|g\|_{L^2(\Gamma_{\epsilon};\omega_{\epsilon})}\epsilon^{\frac 1 2}\|\nabla u^{\epsilon}-\nabla u\|_{L^2(\Gamma_{\epsilon};\omega_{\epsilon})}
			\leq C\epsilon^2+\frac 1 8\|\nabla u^{\epsilon}-\nabla u\|_{L^2(\Gamma_{\epsilon};\omega_{\epsilon})}^2.
		\end{aligned}\right.
	\end{equation}
	
	Combining \eqref{lemma3-3} and \eqref{lemma3-2} together, and applying the $L^2$ norm error estimate in Theorem \ref{thm3}, we easily derive the result \eqref{lemma3-1}.
	
\end{proof}

\begin{lemma}
	\label{lemma4}
	Suppose that $0<\epsilon\leq \epsilon_0$,  $g\in L^{\infty}(0,T;H^1(D_{\epsilon};\omega_{\epsilon}))$, and $g_t\in L^{\infty}(0,T;L^{\infty}(D_{\epsilon}))$. Assume the exact solution $u\in L^2(0,T;H^1(D))$ and its extension  to $D_\epsilon$ fulfills Assumption  \ref{regularity_assump}. If $\kappa \leq A(\bm x) \leq \kappa^{-1}$ for all $\bm x \in D_{\epsilon}$ with some constant $\kappa >0$, then we have
	\begin{equation}
		\label{lemma4-1}
		\begin{aligned}
			&\Big|\int_0^T\int_{D_{\epsilon}}g(u_t^{\epsilon}-u_t)|\nabla \omega_{\epsilon}|\dd\bm x\dd t-\int_0^T \int_{D_{\epsilon}} A\nabla u_t\cdot\nabla \omega_{\epsilon}(u^{\epsilon}-u)\dd \bm x\dd t\\
			&-\int_0^T\int_{\partial D}g(u_t^{\epsilon}-u_t)\dd \sigma\dd t
			-\int_0^T\int_{\partial D}g_t(u^{\epsilon}-u)\dd \sigma\dd t\Big|\\
			&\qquad\leq C\epsilon^2+\frac 1 4\|\nabla u^{\epsilon}(T,\cdot)-\nabla u(T,\cdot)\|_{L^2(\Gamma_{\epsilon};\omega_{\epsilon})}^2+\frac 1 4\|\nabla u^{\epsilon}(0,\cdot)-\nabla u(0,\cdot)\|_{L^2(\Gamma_{\epsilon};\omega_{\epsilon})}^2,
		\end{aligned}
	\end{equation}
	where $C$ is a constant independent of $\epsilon$.
\end{lemma}

\begin{proof}
	Applying with integration by part, we can reformulate the left-hand side of  \eqref{lemma4-1} as
	\begin{equation}\label{Lemma4-2}
		\begin{aligned}
			&\Big|\int_0^T\int_{D_{\epsilon}}g(u_t^{\epsilon}-u_t)|\nabla \omega_{\epsilon}|\dd\bm x\dd t-\int_0^T \int_{D_{\epsilon}} A\nabla u_t\cdot\nabla \omega_{\epsilon}(u^{\epsilon}-u)\dd \bm x\dd t\\
			&-\int_0^T\int_{\partial D}g(u_t^{\epsilon}-u_t)\dd \sigma\dd t-\int_0^T\int_{\partial D}g_t(u^{\epsilon}-u)\dd \sigma\dd t\Big|\\
			&\quad\leq\left|\int_{D_{\epsilon}}g(T,\cdot)(u^{\epsilon}(T,\cdot)-u(T,\cdot))|\nabla\omega_{\epsilon}|\dd \bm x-\int_{\partial D}g(T,\cdot)(u^{\epsilon}(T,\cdot)-u(T,\cdot))\dd\sigma \right|\\
			&\qquad+\left|\int_{D_{\epsilon}}g(0,\cdot)(u^{\epsilon}(0,\cdot)-u(0,\cdot))|\nabla\omega_{\epsilon}|\dd \bm x-\int_{\partial D}g(0,\cdot)(u^{\epsilon}(0,\cdot)-u(0,\cdot))\dd\sigma \right|\\
			&\qquad+\left|\int_0^T\int_{D_{\epsilon}}g_t(u^{\epsilon}-u)|\nabla\omega_{\epsilon}|\dd\bm x\dd t+\int_0^T\int_{D_{\epsilon}}A\nabla u_t\cdot\nabla \omega_{\epsilon}(u^{\epsilon}-u)\dd\bm x\dd t \right|\\
			&\quad=:\rm{III}_1+\rm{III}_2+\rm{III}_3.
		\end{aligned}
	\end{equation}
	In the remaining part, we will estimate the term $\rm{III}_1$, $\rm{III}_2$ and $\rm{III}_3$, respectively. As for the estimate of $\rm{III}_1$ and $\rm{III}_2$, we can apply the Lemma \ref{lemma3} to derive
	\begin{equation}\left\{
		\label{lemma4-3}
		\begin{aligned}
			\rm{III}_1&\leq C\epsilon^2+\frac{1}{4}\|\nabla u^{\epsilon}(T,\cdot)-\nabla u(T,\cdot)\|_{L^2(\Gamma_{\epsilon};\omega_{\epsilon})}^2,\\
			\rm{III_2}&\leq C\epsilon^2+\frac 1 4 \|\nabla u^{\epsilon}(0,\cdot)-\nabla u(0,\cdot)\|_{L^2(\Gamma_{\epsilon};\omega_{\epsilon})}^2.
		\end{aligned}\right.
	\end{equation}
	As for the estimate of $\rm{III}_3$, by using the Cauchy-Schwarz inequality and the definition of $\omega_{\epsilon}$, we have
	\begin{equation}\label{lemma4-4}
		\begin{aligned}
			\rm{III}_3 &= \left|\int_0^T\int_{D_{\epsilon}} g_t(u^{\epsilon}-u)|\nabla\omega_{\epsilon}|\dd\bm x\dd t-\int_0^T\int_{D_{\epsilon}}\bm n \cdot(A\nabla u_t)(u^{\epsilon}-u)|\nabla\omega_{\epsilon}|\dd\bm x\dd t \right|\\
			&=\left|\int_0^T\int_{D_{\epsilon}}(\bm n\cdot(A\nabla u_t)-g_t)(u^{\epsilon}-u)|\nabla \omega_{\epsilon}|\dd\bm x\dd t \right|\\
			&\leq\int_0^T\left(\int_{D_{\epsilon}}\left|\omega_{\epsilon}^{\frac 1 2}(u^{\epsilon}-u) \right|^2\dd\bm x \right)^{\frac 1 2}\left( \int_{D_{\epsilon}}\left|(\bm n\cdot(A\nabla u_t)-g_t)\left|\nabla\omega_{\epsilon}\right|/\omega_{\epsilon}^{\frac 1 2} \right|^2\dd \bm x\right)^{\frac 1 2}\dd t\\
			&\lesssim \sup_{0\leq t\leq T}\|u^{\epsilon}(t)-u(t)\|_{L^2(D_{\epsilon};\omega_{\epsilon})}\|\bm n\cdot(A\nabla u_t)-g_t\|_{L^{\infty}(D_{\epsilon})}\\
			&\lesssim \epsilon^2.
		\end{aligned}
	\end{equation}
	Inserting \eqref{lemma4-3} and \eqref{lemma4-4} to \eqref{Lemma4-2}, we finally obtain \eqref{lemma4-1}, which completes the proof.
	
\end{proof}


\begin{theorem}[Error estimate in the $H^1$-norm]
	\label{thm5}
	Suppose that $0<\epsilon\leq \epsilon_0$,  $g\in L^{\infty}(0,T;H^2(D_{\epsilon};\omega_{\epsilon}))$, $g_t\in L^{\infty}(0,T;L^{\infty}(D_{\epsilon}))$,
	$\kappa\leq A(\bm x)\leq \kappa^{-1}$ for all $\bm x \in D$ with some constant $\kappa>0$, and $f$ satisfies  Assumption \ref{assumption1}. Assume the exact solution $u\in L^2(0,T;H^1(D_{\epsilon}))$ and its extension to $D_\epsilon$ fulfills Assumption  \ref{regularity_assump}. Then we have
	\begin{equation}
		\label{thm5-1}
		\|u^{\epsilon}(t)-u(t)\|_{H^1(D_{\epsilon};\omega_{\epsilon})}\lesssim \epsilon, \quad \forall\ 0 \leq t\leq T,
	\end{equation}
	where the hidden constant is independent of $\epsilon$.
	
\end{theorem}

\begin{proof}
	First, we insert $v=2(u^{\epsilon}_t-u_t)$ into \eqref{eq4-1}, we have
	\begin{equation}
		\begin{aligned}
			\label{thm5-2}
			&2\left(u_t^{\epsilon}-u_t,u_t^{\epsilon}-u_t\right)_{D_{\epsilon},\omega_{\epsilon}}+2a^{\epsilon}(u^{\epsilon}-u,u_t^{\epsilon}-u_t)\\
			&\quad = 2\left(u_t,u_t^{\epsilon}-u_t\right)-2\left(u_t,u_t^{\epsilon}-u_t\right)+2a(u,u_t^{\epsilon}-u_t)-2a^{\epsilon}(u,u_t^{\epsilon}-u_t)\\
			&\quad \quad +2\ell^{\epsilon}(u_t^{\epsilon}-u_t)-2\ell(u_t^{\epsilon}-u_t).
		\end{aligned} 
	\end{equation}
	
	Then, after integrating both sides of \eqref{thm5-2} with respect to $t$, we can derive that
	\begin{equation}
		\begin{aligned}
			\label{thm5-3}
			&\int_0^T2\left\|u_t^{\epsilon}-u_t\right\|_{L^2(D_{\epsilon};\omega_{\epsilon})}^2+\frac{\dd}{\dd t}a^{\epsilon}(u^{\epsilon}-u,u^{\epsilon}-u)\dd t\\
			&\quad \leq\int_0^T\left(u_t,u_t^{\epsilon}-u_t\right)-\left(u_t,u_t^{\epsilon}-u_t\right)_{D_{\epsilon};\omega_{\epsilon}}\dd t+2\int_0^T a(u,u_t^{\epsilon}-u_t)-a^{\epsilon}(u,u_t^{\epsilon}-u_t)\dd t\\
			&\quad \quad+2\int_0^T\ell^{\epsilon}(u_t^{\epsilon}-u_t)-\ell(u_t^{\epsilon}-u_t)\dd t\\
			&\quad=:\rm{IV}_1+\rm{IV}_2+\rm{IV}_3.
		\end{aligned}
	\end{equation}
	Next, we will estimate $\rm{IV}_1$, $\rm{IV}_2$ and $\rm{IV}_3$ respectively. For $\rm{IV}_1$, applying with Lemma \ref{lemma1}, we have
	\begin{equation*}
		\begin{aligned}
			\label{thm5-4}
			\left|\rm{IV}_1\right| &= 2\left|\int_0^T \left(u_t,u_t^{\epsilon}-u_t\right)-\left(u_t,u_t^{\epsilon}-u_t\right)_{D_{\epsilon};\omega_{\epsilon}} \dd t\right|\\
			&\lesssim \int_0^T\left|\left(u_t,u_t^{\epsilon}-u_t\right)-\left(u_t,u_t^{\epsilon}-u_t\right)_{D_{\epsilon};\omega_{\epsilon}} \right|\dd t\\
			&\lesssim \int_0^T\left\|u_t^{\epsilon}-u_t\right\|_{L^2(\Gamma_{\epsilon};\omega_{\epsilon})}\epsilon^{\frac 1 2}\|u_t\|_{L^2(\Gamma_{\epsilon};\omega_{\epsilon})}\dd t.
		\end{aligned}
	\end{equation*}
	Due to the fact that $u^{\epsilon}$ and $u_t^{\epsilon}$, as well as $u$ and $u_t$, have the same regularity, based on Theorem \ref{thm3}, we get that
	\begin{align}
		\label{thm5-5}
		|\rm{IV}_1| &\lesssim \int_0^T \epsilon^2\epsilon^{\frac 1 2}\|u_t\|_{\Gamma_{\epsilon};\omega_{\epsilon}}\dd t\lesssim \epsilon^3.
	\end{align}
	
	As for $\rm{IV}_2$, with the integration by part, we derive that
	\begin{equation}
		\begin{aligned}
			\label{thm5-6}
			\rm{IV}_2=&\,2\int_0^T a(u,u_t^{\epsilon}-u_t)-a^{\epsilon}(u,u_t^{\epsilon}-u_t)\dd t \\
			=&\,2\int_D A\Big( \nabla u(T)\cdot\left(\nabla u^{\epsilon}(T)-\nabla u(T)\right)-\nabla u(0,\cdot)\cdot\left(\nabla u^{\epsilon}(0)-\nabla u(0)\right)\\
			&-\int_0^T (\nabla u^{\epsilon}-\nabla u)\cdot\nabla u_t\dd t\Big)\dd \bm x-2\int_D A\Big( \nabla u(T)\cdot\left(\nabla u^{\epsilon}(T)-\nabla u(T)\right)\\
			&-\nabla u(0)\cdot\left(\nabla u^{\epsilon}(0)-\nabla u(0)\right)-\int_0^T (\nabla u^{\epsilon}-\nabla u)\cdot\nabla u_t\dd t\Big)\dd \bm x\\
			=&\,2\Big(\int_DA\nabla u(T)\cdot(\nabla u^{\epsilon}(T)-\nabla u(T))\dd\bm x-\int_{D_{\epsilon}} A\nabla u(T)\\
			&\cdot(\nabla u^{\epsilon}(T)-\nabla u(T))\omega_{\epsilon}\dd\bm x\Big) 
			+2\Big(\int_{D_{\epsilon}}A\nabla u(0)\cdot(\nabla u^{\epsilon}(0)-\nabla u(0))\omega_{\epsilon}\dd\bm x\\
			&-\int_D A\nabla u(0)\cdot(\nabla u^{\epsilon}(0)-\nabla u(0))\dd\bm x\Big) +2\Big(\int_0^T\Big(\int_{D_{\epsilon}} A(\nabla u^{\epsilon}-\nabla u)\\
			&\cdot\nabla u_t\omega_{\epsilon}\dd\bm x-\int_D A(\nabla u^{\epsilon}-\nabla u)\cdot\nabla u_t\dd\bm x\Big)\dd t\Big)\\
			=:&\,\rm{V}_1+\rm{V}_2+\rm{V}_3.
		\end{aligned}
	\end{equation}
	For $\rm{V}_1$, applying with Lemma \ref{lemma1}, we know that
	\begin{equation}
		\begin{aligned}
			\label{thm5-7}
			|\rm{V}_1|&\lesssim\epsilon^{\frac 1 2}\|\nabla u^{\epsilon}(T)-\nabla u(T)\|_{L^2(\Gamma_{\epsilon};\omega_{\epsilon})}\|A\nabla u(T)\|_{L^2(\Gamma_{\epsilon};\omega_{\epsilon})}\\
			&\leq \frac 1 4 \|\nabla u^{\epsilon}(T)-\nabla u(T)\|_{L^2(\Gamma_{\epsilon};\omega_{\epsilon})}^2+C\epsilon \|A\nabla u(T)\|_{L^2(\Gamma_{\epsilon};\omega_{\epsilon})}^2\\
			&\leq \frac 1 4 \|\nabla u^{\epsilon}(T)-\nabla u(T)\|_{L^2(\Gamma_{\epsilon};\omega_{\epsilon})}^2+C\epsilon^2,
		\end{aligned}
	\end{equation}
	Similar as the derivation of \eqref{thm5-7}, we have
	\begin{align}
		\label{thm5-8}
		|\rm{V}_2|&\lesssim \epsilon^\frac{1}{2}\|\nabla u^{\epsilon}(0)-\nabla u(0)\|_{L^2(\Gamma_{\epsilon};\omega_{\epsilon})}\|A\nabla u(0)\|_{L^2(\Gamma_{\epsilon};\omega_{\epsilon})}\lesssim \epsilon^2,
	\end{align}
	with proper initial guess for $u^{\epsilon}$. For the remaining term $\rm{V}_3$, applying with the variational formula and the property of $\omega_{\epsilon}$, we can rewrite $\rm{V}_3$ as
	\begin{equation}
		\begin{aligned}
			\label{thm5-9}
			\rm{V}_3=&\,2\left(\int_0^T\left(\int_{D_{\epsilon}} A(\nabla u^{\epsilon}-\nabla u)\cdot\nabla u_t\omega_{\epsilon}\dd\bm x-\int_D A(\nabla u^{\epsilon}-\nabla u)\cdot\nabla u_t\dd\bm x\right)\dd t\right)\\
			=&\,2\int_0^T\int_{\partial D_{\epsilon}}\bm n\cdot(A\nabla u_t)\omega_{\epsilon}(u^{\epsilon}-u)\dd \sigma-\int_{D_{\epsilon}}\divo(A\nabla u_t\omega_{\epsilon})(u^{\epsilon}-u)\dd\bm x\\
			&-\int_{\partial D}\bm n\cdot(A\nabla u_t)(u^{\epsilon}-u)\dd\sigma+\int_D \divo(A\nabla u_t)(u^{\epsilon}-u)\dd\bm x\dd t\\
			=&\,2\left(\int_0^T\int_D \divo(A\nabla u_t)(u^{\epsilon}-u)\dd\bm x-\int_{D_{\epsilon}}\divo(A\nabla u_t)(u^{\epsilon}-u)\omega_{\epsilon}\dd\bm x\dd t\right)\\
			&-2\left(\int_0^T\int_{D_{\epsilon}} A\nabla u_t\cdot\nabla \omega_{\epsilon}(u^{\epsilon}-u)\dd\bm x+\int_{\partial D}g_t(u^{\epsilon}-u)\dd\sigma\dd t\right)\\
			=:&\,\rm{W}_1-\rm{W}_2.
		\end{aligned}
	\end{equation}

	Following the similar derivation of \eqref{thm5-7}, we obtain
	\begin{equation}
		\begin{aligned}
			\label{thm5-10}
			|\rm{W}_1|&\lesssim \epsilon^{\frac 1 2}\|u^{\epsilon}-u\|_{L^2(\Gamma_{\epsilon};\omega_{\epsilon})}\|\divo(A\nabla u_t)\|_{L^2(\Gamma_{\epsilon};\omega_{\epsilon})}\lesssim 
			\epsilon^2.
		\end{aligned}
	\end{equation}
	The terms remained to be analyzed are $\rm{W}_2$ and $\rm{IV}_3$. Merging $\rm{W}_2$ and $\rm{IV}_3$ together, we have
	\begin{equation}
		\begin{aligned}
			\label{thm5-11}
			|-\rm{W}_2+\rm{IV}_3|\leq&\,2\left|\int_0^T\int_{D_{\epsilon}}f(u_t^{\epsilon}-u_t)\omega_{\epsilon}\dd\bm x-\int_D f(u_t^{\epsilon}-u_t)\dd\bm x\dd t\right|\\
			&+2\left|\int_0^T\int_{D_{\epsilon}} g(u_t^{\epsilon}-u_t)|\nabla\omega_{\epsilon}|\dd\bm x-\int_{\partial D} g(u_t^{\epsilon}-u_t)\dd\sigma\dd t\right.\\
			&\left.-\int_0^T\int_{D_{\epsilon}} A\nabla u_t\cdot\nabla \omega_{\epsilon}(u^{\epsilon}-u)\dd\bm x+\int_{\partial D}g_t(u^{\epsilon}-u)\dd\sigma\dd t\right|.
		\end{aligned}
	\end{equation}
	
	By using the Lemma \ref{lemma1}, we can easily derive that
	\begin{equation}
		\begin{aligned}
			\label{thm5-12}
			&\left|\int_0^T\int_{D_{\epsilon}}f(u_t^{\epsilon}-u_t)\omega_{\epsilon}\dd\bm x-\int_D f(u_t^{\epsilon}-u_t)\dd\bm x\dd t\right|\\
			&\qquad\lesssim\int_0^T \epsilon^{\frac 1 2}\|f\|_{L^2(\Gamma_{\epsilon};\omega_{\epsilon})}\|u_t^{\epsilon}-u_t\|_{L^2(\Gamma_{\epsilon};\omega_{\epsilon})}\dd t\lesssim \epsilon^3.
		\end{aligned}
	\end{equation}	
	By using the Lemma \ref{lemma4}, we can derive that
	\begin{equation}
		\begin{aligned}
			\label{thm5-13}
			&\left| \int_0^T\int_{D_{\epsilon}} g(u_t^{\epsilon}-u_t)|\nabla\omega_{\epsilon}|\dd\bm x\dd t-\int_0^T\int_{D_{\epsilon}}A\nabla u_t\cdot\nabla\omega_{\epsilon}(u^{\epsilon}-u)\dd\bm x\dd t\right.\\
			&\left.-\int_0^T\int_{\partial D}g(u_t^{\epsilon}-u_t)\dd\sigma\dd t
			-\int_0^T\int_{\partial D} g_t(u^{\epsilon}-u)\dd\sigma\dd t\right|\\
			&\qquad \leq C\epsilon^2+\frac{1}{4}\|\nabla u^{\epsilon}(T)-\nabla u(T)\|_{L^2(\Gamma_{\epsilon};\omega_{\epsilon})}^2+\frac 1 4 \|\nabla u^{\epsilon}(0)-\nabla u(0)\|_{L^2(\Gamma_{\epsilon};\omega_{\epsilon})}^2.
		\end{aligned}
	\end{equation}
	Combining \eqref{thm5-12} and \eqref{thm5-13} together, we can reformulate \eqref{thm5-11} as
	{\small\begin{equation}
			\begin{aligned}
				\label{thm5-14}
				|-\rm{W}_2+\rm{IV}_3|\leq&\, C\epsilon^2+\frac 1 4\|\nabla u^{\epsilon}(T)-\nabla u(T)\|_{L^2(\Gamma_{\epsilon};\omega_{\epsilon})}^2+\frac 1 4 \|\nabla u^{\epsilon}(0)-\nabla u(0)\|_{L^2(\Gamma_{\epsilon};\omega_{\epsilon})}^2.
			\end{aligned}
	\end{equation}}
	
	Therefore, inserting \eqref{thm5-5}-\eqref{thm5-14} to \eqref{thm5-3}, we can conclude that
	{\small	\begin{equation*}
			\begin{aligned}
				&\int_0^T 2\|u_t^{\epsilon}-u_t\|_{L^2(D_{\epsilon};\omega_{\epsilon})}^2\dd t+\|\nabla u^{\epsilon}(T)-\nabla u(T)\|_{L^2(D_{\epsilon};\omega_{\epsilon})}^2-\|\nabla u^{\epsilon}(0)-\nabla u(0)\|_{L^2(D_{\epsilon};\omega_{\epsilon})}^2\\
				&\qquad\leq C\epsilon^2+\frac 1 2\|\nabla u^{\epsilon}(T)-\nabla u(T)\|_{L^2(D_{\epsilon};\omega_{\epsilon})}^2+\frac 1 4\|\nabla u^{\epsilon}(0)-\nabla u(0)\|_{L^2(D_{\epsilon};\omega_{\epsilon})}^2,
			\end{aligned}
	\end{equation*}}
	which implies that
	\begin{equation*}
		\|\nabla u^{\epsilon}(T)-\nabla u(T)\|_{L^2(D_{\epsilon};\omega_{\epsilon})}^2\lesssim \epsilon^2 + \|\nabla u^{\epsilon}(0)-\nabla u(0)\|_{L^2(D_{\epsilon};\omega_{\epsilon})}^2.
	\end{equation*}
	Since $T$ can be replaced by any time $t\in[0,T]$ in the above derivations, we get the conclusion \eqref{thm5-1}.
	
\end{proof}

\section{Numerical experiments}\label{numerical}

In this section, we will present some numerical experiments to verify the error estimates (Theorems \ref{thm3}  and Theorem \ref{thm5} ) obtained in Section \ref{theory} and demonstrate the performance of the DDM. We apply the finite element method with the bilinear basis functions  for space discretization and  the BDF2 scheme for time stepping to solve the 
DDM-transformed problem \eqref{weighted_variational} defined on a larger rectangular domain, which is of second-order  in time, and first-order  in space with respect to the $H^1$ norm and second-order  in space with respect to the $L^2$ norm.
To demonstrate  the convergence order of the diffuse domain solution with respect to the interface width parameter $\epsilon$, i.e., $\left\|u(t)-u^{\epsilon}(t)\right\|_{L^2(D_{\epsilon};\omega_{\epsilon})}$ and $\left\|u(t)-u^{\epsilon}(t)\right\|_{H^1(D_{\epsilon};\omega_{\epsilon})}$, we will
take  $\epsilon$ to be much larger than the spatial mesh size and the time step size (so that the solution errors are mainly caused by the DDM approximation). For simplicity, we only evaluate the approximation error on the original domain $D$, i.e.,  $\left\|u(t)-u^{\epsilon}(t)\right\|_{L^2(D;\omega_{\epsilon})}$ and $\left\|u(t)-u^{\epsilon}(t)\right\|_{H^1(D;\omega_{\epsilon})}$. 
All tests are done using Matlab on a laptop with Intel Ultra 9 185H, 2.30GHz CPU and 32GB memory.

\subsection{The case of constant diffusion coefficient}


\begin{example}
	\label{ex1}
	In this example, we consider the following two-dimensional diffusion problem with a constant diffusion coefficient: for $0 \leq t \leq T$, 
	\begin{equation}
		\left\{\begin{aligned}
			&u_t=3\Delta u +f(t, x, y), \ &&(x,y)\in D\\
			&u(0,x,y) = \left(x^2+2x\right)\left(y^2-2y \right), \ \  &&(x, y) \in D,
		\end{aligned}
		\right.
	\end{equation}
	where
	\begin{equation*}
		\begin{split}
			f(t, x, y) =& -e^{-\pi^2 t}\left(\pi^2\left(x^2+2x\right)\left(y^2-2y\right)+6\left( x^2+2x\right)+6\left(y^2-2y\right) \right).
		\end{split}
	\end{equation*}
	The exact solution is given by $$u(t,x,y)=e^{-\pi^2 t}\left(x^2+2x \right)\left(y^2-2y \right).$$
	Two different domains $D$ are considered: one is a circular domain defined by
	\begin{equation}
		\label{circle}
		D=\left\{(x,y)\;\big|\; x^2+y^2=\frac{1}{16} \right\},
	\end{equation}
	and the other is a flower-shaped domain defined by
	\begin{equation}
		\label{flower}
		D=\left\{(x,y)\;\big|\; x^2+y^2-\left(0.175-0.03\sin\left(4\arctan\left(\frac y x\right)\right)\right)^2=0 \right\}.
	\end{equation}
	The Neumann boundary condition is imposed correspondingly, and 
	the terminal time is chosen as $T=0.5$. 
\end{example}

We run the diffuse domain method (DDM) on the extended rectangular domain $\Omega=[-1/2,1/2]\times[-1/2, 1/2]$ and $D\subset\Omega$. We take the time steps $N_T=512$ (i.e., $\Delta \tau = T/N_T=1/1024$ and the uniformly spatial mesh with  $h_x = h_y=1/512$. Then we set the interface thickness $\epsilon$ = 1/8, 1/16, 1/32, 1/64, respectively, so that the spatial mesh size and temporal step size are much finer than the interface thickness $\epsilon$. All numerical results are reported in Table \ref{tab 1}, including the solution errors measured in the weighted $L^2$ and $H^1$ norms and the corresponding convergence rates. We observe  roughly second-order convergence rates with respect to the weighted $L^2$ norm and first-order  convergence rates with respect to the weighted $H^1$ norm as expected, which coincide very well with the error estimates derived in Theorems \ref{thm3} and \ref{thm5}. Figure \ref{fig 1} presents simulated phase structures of the numerical solutions at the terminal time with the interface thickness $\epsilon=1/16,1/32,1/64$, respectively. We clearly observe that as the interface thickness decreases, the transition zone gradually becomes narrower and the shape of the approximated region increasingly resembles the exact region. The inner phase structure depends on the intensity of fluctuations in the exact solution $u$ and the phase-field function $\varphi$. Numerical errors tend to be more significant near the domain boundaries and in areas where the objective function changes rapidly, whereas in the interior of the domain and in regions where the function varies smoothly, numerical errors tends to be smaller. 

\begin{table}[!htbp]
	\centering
	\caption{Numerical results on the solution errors measured in the weighted $L^2$ and $H^1$ norms and corresponding convergence rates at the terminal time $T=0.5$ produced by the DDM in Example \ref{ex1}.}
	\begin{tabular}{|ccccc|}
		\hline
		$\epsilon$ & $\|u^{\epsilon}-u(t_n)\|_{L^2(D;\omega_{\epsilon})}$ & CR & $\|u^{\epsilon}-u(t_n)\|_{H^1(D;\omega_{\epsilon})}$ & CR\\
		\hline
		\multicolumn{5}{|c|}{Convergence tests for the circular domain}\\
		\hline
		1/8 &  1.0000e-03 & - & 1.0000e-03 & - \\
		1/16 & 2.6803e-04 & 1.90 & 2.7212e-04 & 1.88 \\
		1/32 & 6.8145e-05 & 1.98 & 7.6650e-05 & 1.83\\
		1/64 & 1.7663e-05 & 1.95 & 3.8208e-05 & 1.00 \\
		\hline
		\multicolumn{5}{|c|}{Convergence tests for the flower-shaped domain}\\
		\hline
		1/8 & 7.7428e-04 & - & 7.8951e-04 & - \\
		1/16 & 2.1302e-04 & 1.86 & 2.1599e-04 & 1.87 \\
		1/32 & 5.4320e-05 & 1.97 & 5.9805e-05 & 1.85 \\
		1/64 & 1.3876e-05 & 1.97 & 2.7572e-05 & 1.12 \\
		\hline
	\end{tabular}
	\label{tab 1}
\end{table}

\begin{figure}[!htbp]
	\centering
	\subfigure{
		\label{fig1-1}
		\centering
		\includegraphics[width = 110pt,height=100pt]{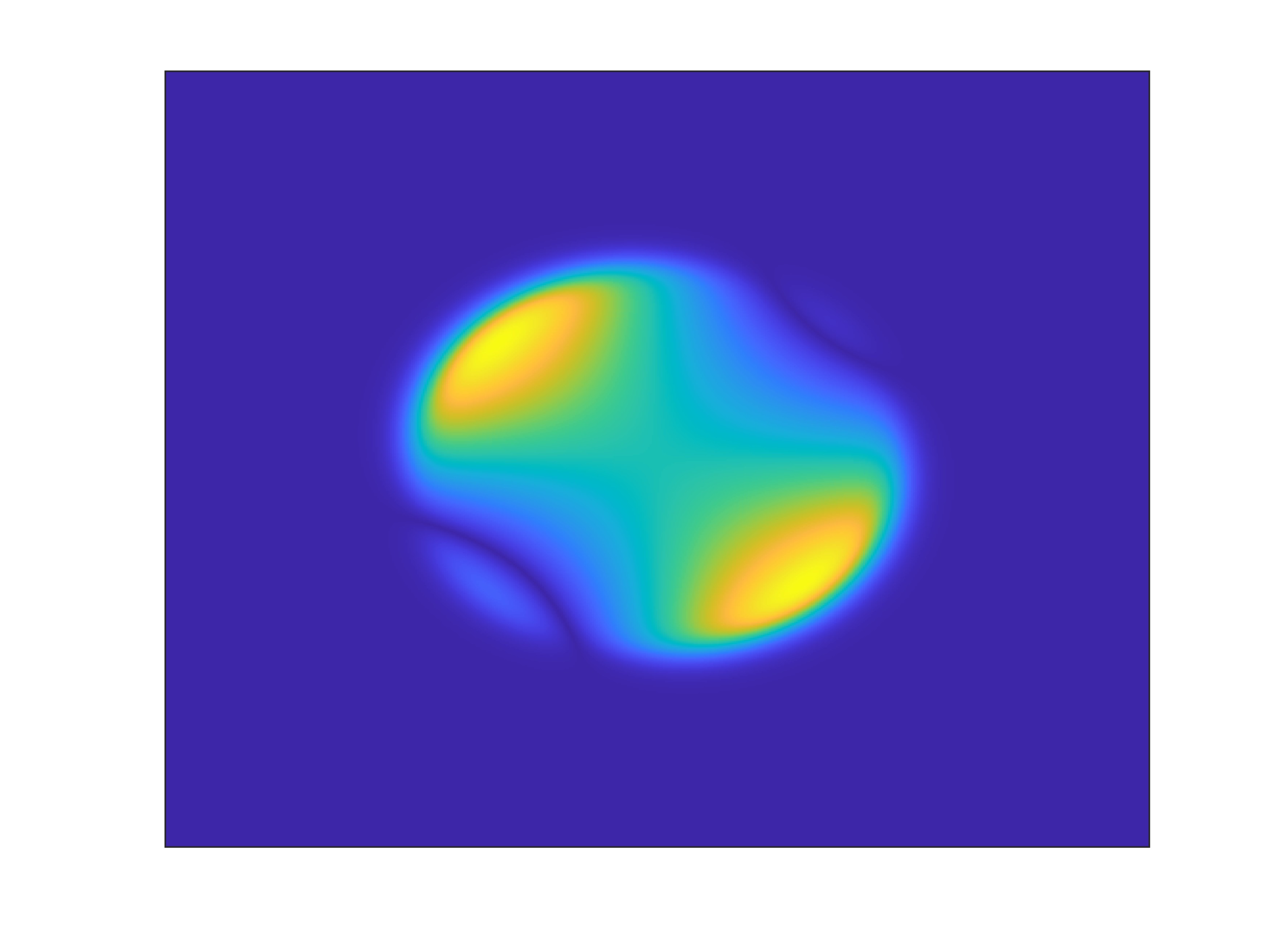}
	}\hspace{-0.6cm}
	\subfigure{
		\label{fig1-2}
		\centering
		\includegraphics[width = 110pt,height=100pt]{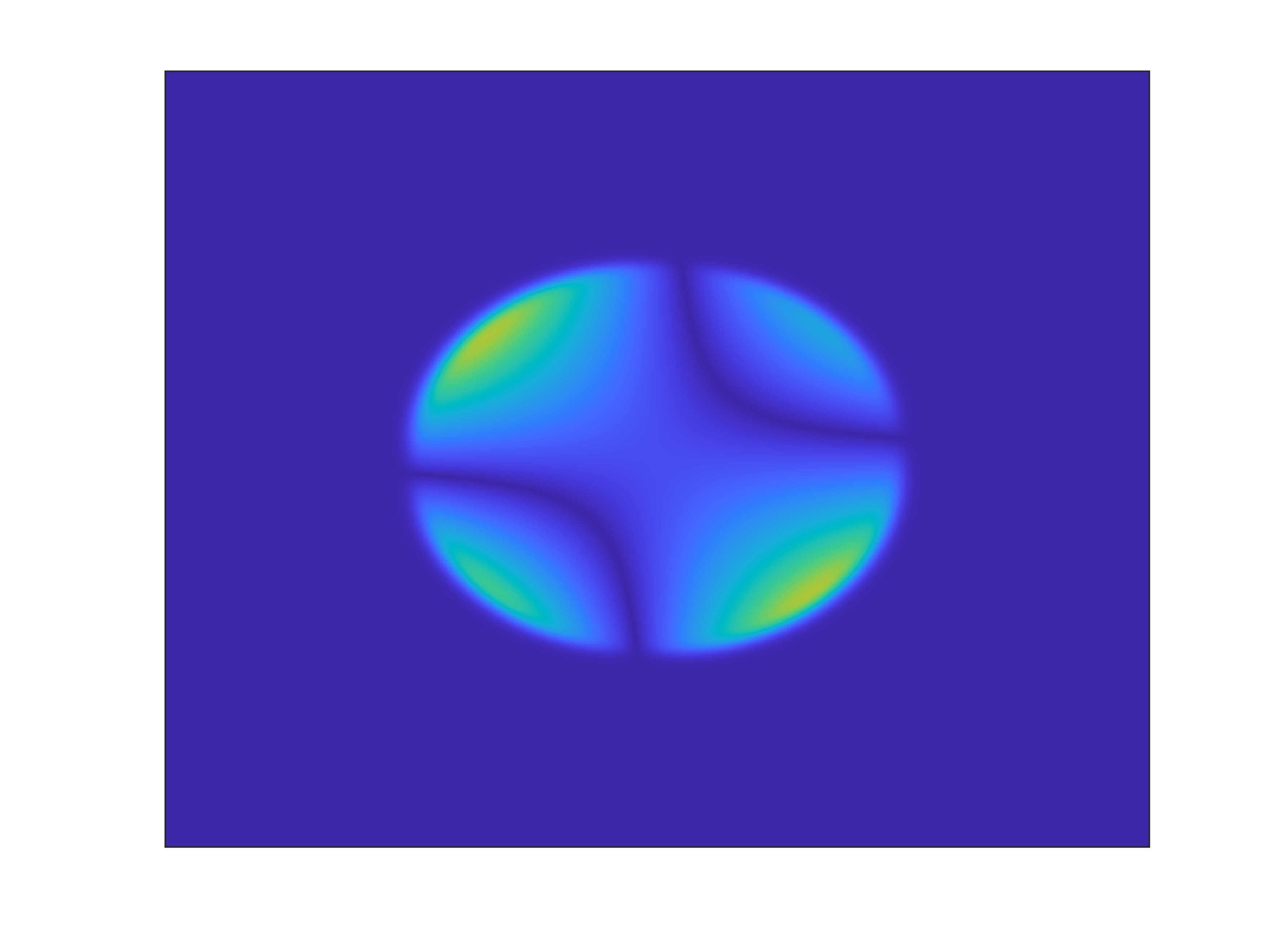}
	}\hspace{-0.6cm}
	\subfigure{
		\label{fig1-3}
		\centering
		\includegraphics[width = 120pt,height=100pt]{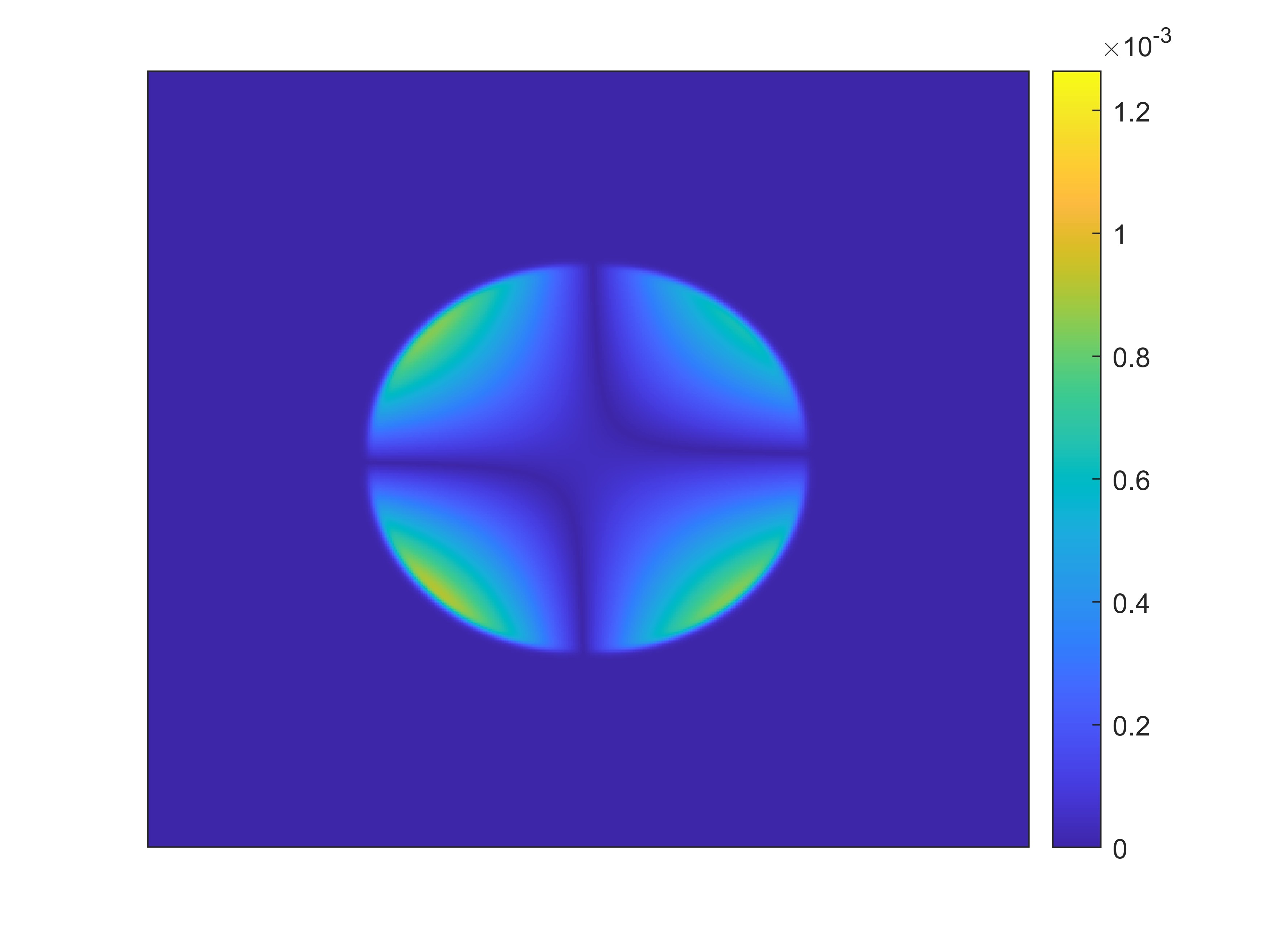}
	}
	
	\subfigure{
		\label{fig4-1}
		\centering
		\includegraphics[width = 110pt,height=100pt]{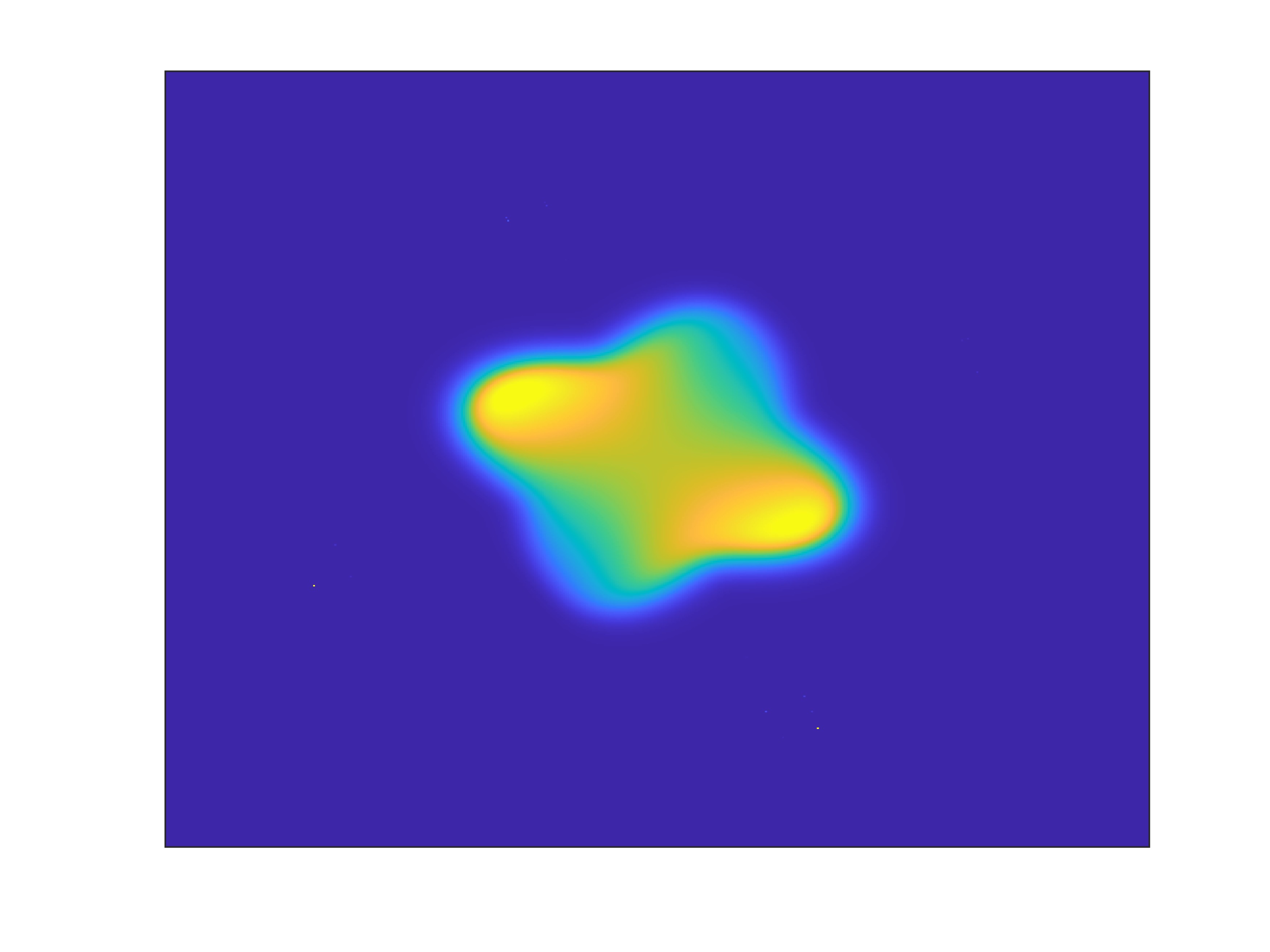}
	}\hspace{-0.6cm}
	\subfigure{
		\label{fig4-2}
		\centering
		\includegraphics[width = 110pt,height=100pt]{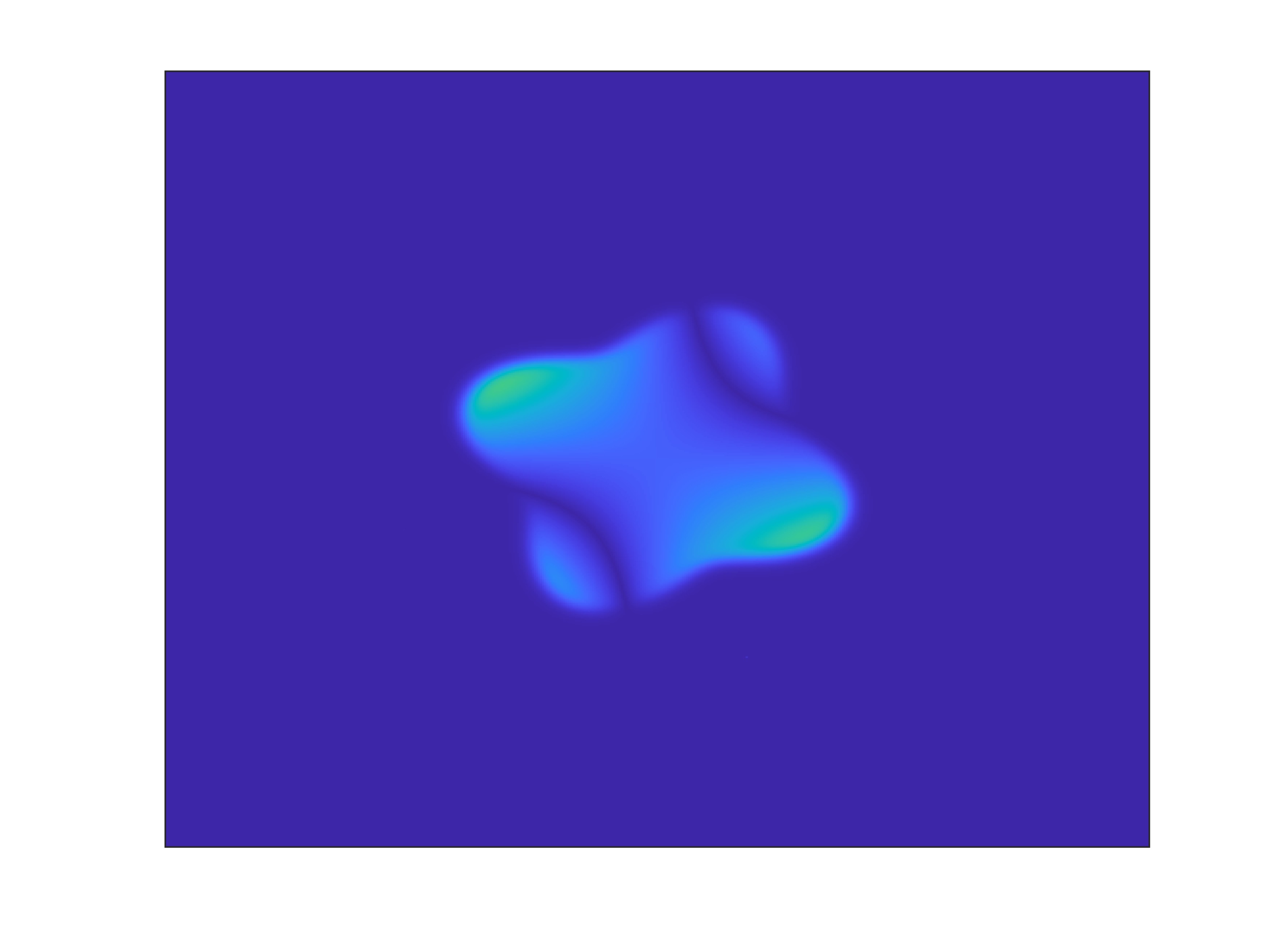}
	}\hspace{-0.6cm}
	\subfigure{
		\label{fig4-3}
		\centering
		\includegraphics[width = 120pt,height=100pt]{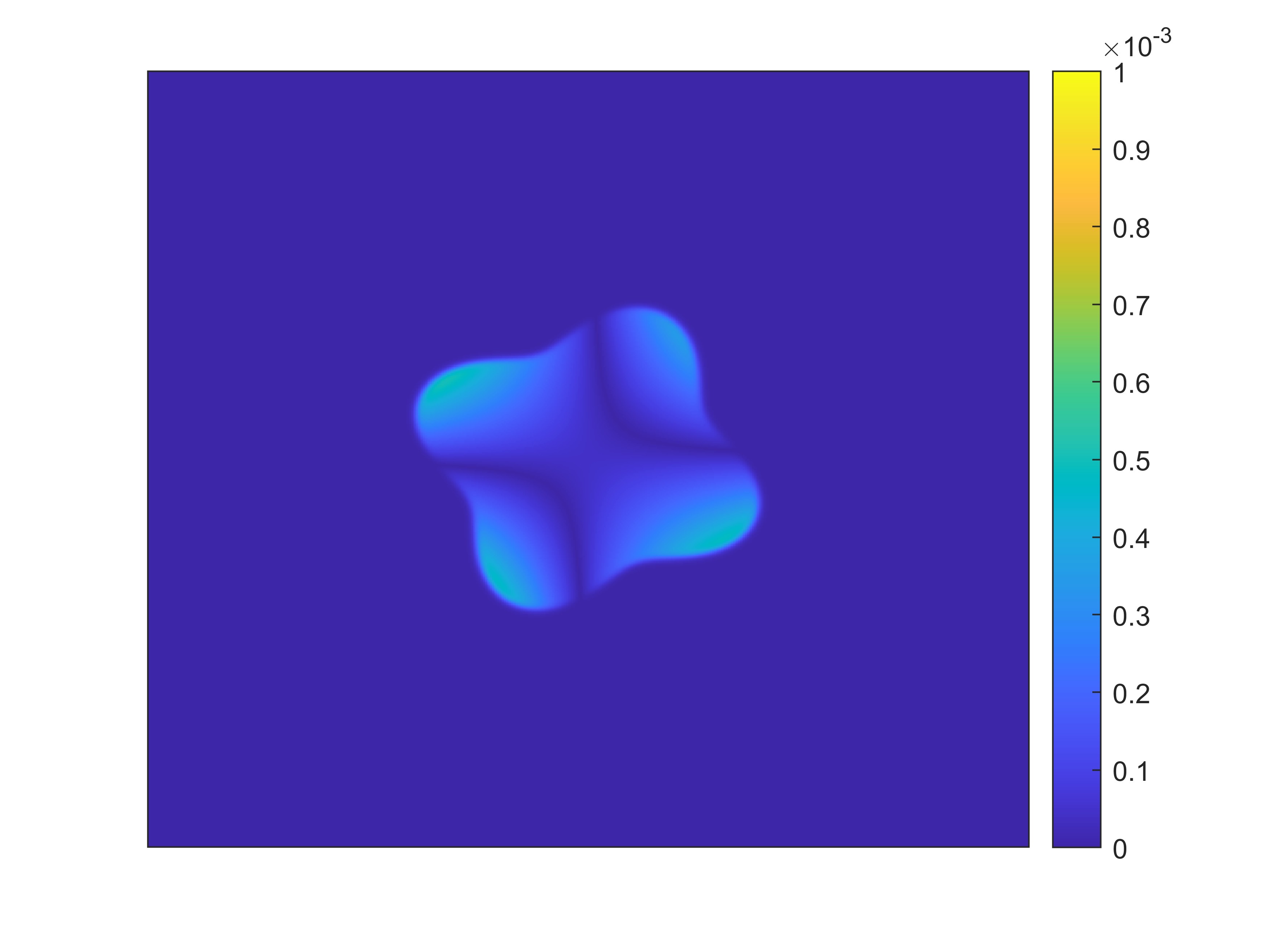}
	}
	\vspace{-0.5cm}
	\caption{The phase structures of the numerical solutions at the terminal time $T=0.5$ produced by the DDM approach with interface thickness $\epsilon=1/16,1/32,1/64$ (from left to right) for Example \ref{ex1}  in the circular domain (top row) and the flower-shaped domain (bottom row).}
	\label{fig 1}
\end{figure}

\subsection{The case of varying diffusion coefficient}


\begin{example}
	\label{ex2}
	In this example, we consider the following two-dimensional  diffusion problem  in the same circular domain and the flower-shaped domain used in Example \ref{ex1},  but with varying diffusion coefficient: for $0 \leq t \leq T$, 
	\begin{equation*}
		\left\{\begin{split}
			&u_t=\nabla \cdot \left((x^2+y^2+3)\nabla u\right)+f(t,x,y), \ &&(x,y)\in D\\
			&u(0,x,y) = \left(2 x^2-4x\right)\left(2 y^2-4y \right), \ \  &&(x, y) \in D,
		\end{split}
		\right.
	\end{equation*}
	where 
	\begin{equation*}
		\begin{split}
			f(t,x,y) =& -e^{-\pi^2 t}\Big(\pi^2\left(2 x^2-4x\right)\left(2 y^2-4y\right)+\left(4(x^2+y^2+3)+2x(4x-4)\right)\\
			&\cdot\left(2 y^2-4y \right)+\left(4(x^2+y^2+3)+2y(4y-4)\right)\left(2 x^2-4x \right)\Big),
		\end{split}
	\end{equation*}
	In this case, the exact solution is given by $$u(t,x,y)=e^{-\pi^2 t}\left(2 x^2-4x \right)\left(2 y^2-4y \right).$$
	Circular domain and flower-shaped domain defined in \eqref{circle} and \eqref{flower} are considered again. The Neumann boundary condition is correspondingly imposed and the terminal time is  set to $T=0.5$.
\end{example}

We verify the approximation accuracy of DDM by fixing $N_T=512$ (i.e., $\Delta \tau = T/N_T=1/1024$) and the extended rectangular domain $\Omega=[-1/2,1/2]\times[-1/2,1/2]$ ($D \subset \Omega$) with uniform spatial meshes, and the corresponding mesh size is $h_x=h_y=1/512$. To guarantee the spatial mesh sizes and temporal step size are much finer than the interface thickness $\epsilon$, we set $\epsilon=1/8,1/16,1/32, 1/64$ respectively for approximation accuracy tests. Table \ref{tab 2} shows the solution errors measured in the weighted $L^2$ and $H^1$ norms, including the corresponding convergence rates. From the Table \ref{tab 2}, we can conclude that the DDM method exhibits second-order convergence rate in the weighted $L^2$ norm, and first-order convergence rate in the weighted $H^1$ norm. The numerical results match very well with the error estimates derived in Theorem \ref{thm3} and \ref{thm5}. Fixing the same spatial meshes and temporal partitions, the simulated phase structures of the numerical solutions at the terminal time are shown in the Figure \ref{fig 2} with interface thickness $\epsilon=1/16,1/32,1/64$, respectively. It's observed again that with the interface thickness decreasing, the transition zone gradually becomes narrower and narrower, and the shape of the approximated region resembles the original irregular domain better and better. Similar as the Example \ref{ex1}, the phase structure of the numerical solution mainly depends on the intensity of fluctuations in the exact solution $u$ and the phase-field function $\varphi$. The numerical error tends to be more significant near the domain boundaries and in areas where the objective function changes rapidly, while that tends to be smaller in the interior of the domain due to the smoothness of the function. 

\begin{table}[!htbp]
	\centering
	\caption{Numerical results on the solution solutions measured in the weighted $L^2$ and $H^1$ norms and corresponding convergence rates at the terminal time $T=0.5$ produced by the DDM in Example \ref{ex2}.}
	\begin{tabular}{|ccccc|}
		\hline
		$\epsilon$ & $\|u^{\epsilon}-u(t_n)\|_{L^2(D;\omega_{\epsilon})}$ & CR & $\|u^{\epsilon}-u(t_n)\|_{H^1(D;\omega_{\epsilon})}$ & CR\\
		\hline
		\multicolumn{5}{|c|}{Approximation tests for circle domain}\\
		\hline
		1/8 &  4.4000e-03 & - & 4.4000e-03 & - \\
		1/16 & 1.1000e-03  & 2.00 & 1.2000e-03 & 1.87 \\
		1/32 & 2.8780e-04  & 1.93 & 3.2838e-04 & 1.87 \\
		1/64 & 7.5129e-05  & 1.94 & 1.6267e-04 & 1.01\\
		\hline
		\multicolumn{5}{|c|}{Approximation tests for flower-shaped domain}\\
		\hline
		1/8 & 3.3000e-03 & - & 3.4000e-03 & - \\
		1/16 & 8.6685e-04 & 1.93 & 8.8274e-04 & 1.95 \\
		1/32 & 2.2363e-04 & 1.95 & 2.4955e-04 & 1.82 \\
		1/64 & 5.7174e-05 & 1.97 & 1.1470e-04 & 1.12 \\
		\hline
	\end{tabular}
	\label{tab 2}
\end{table}

\begin{figure}[!htbp]
	\centering
	\subfigure{
		\label{fig2-1}
		\centering
		\includegraphics[width = 110pt,height=100pt]{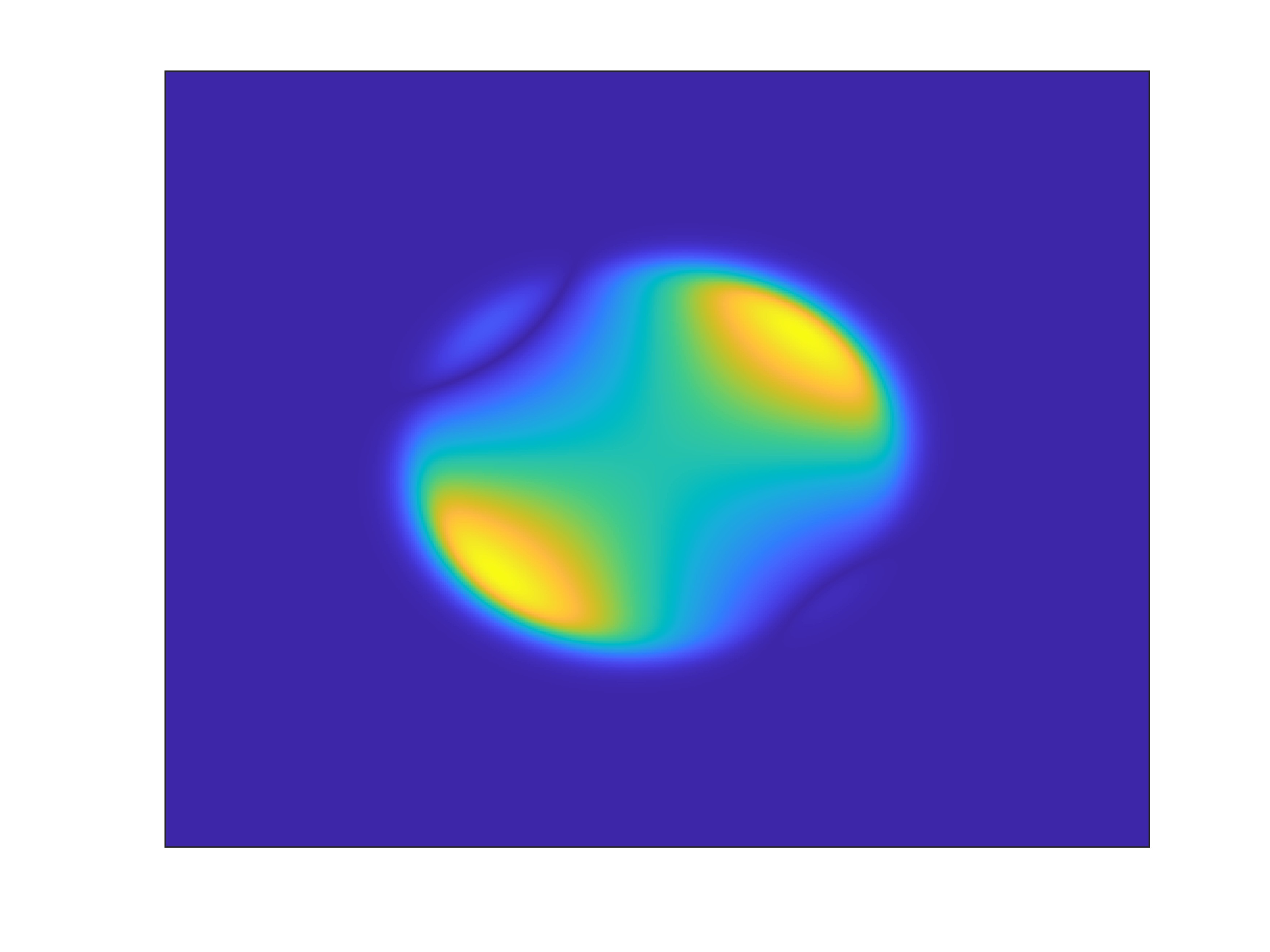}
	}\hspace{-0.6cm}
	\subfigure{
		\label{fig2-2}
		\centering
		\includegraphics[width = 110pt,height=100pt]{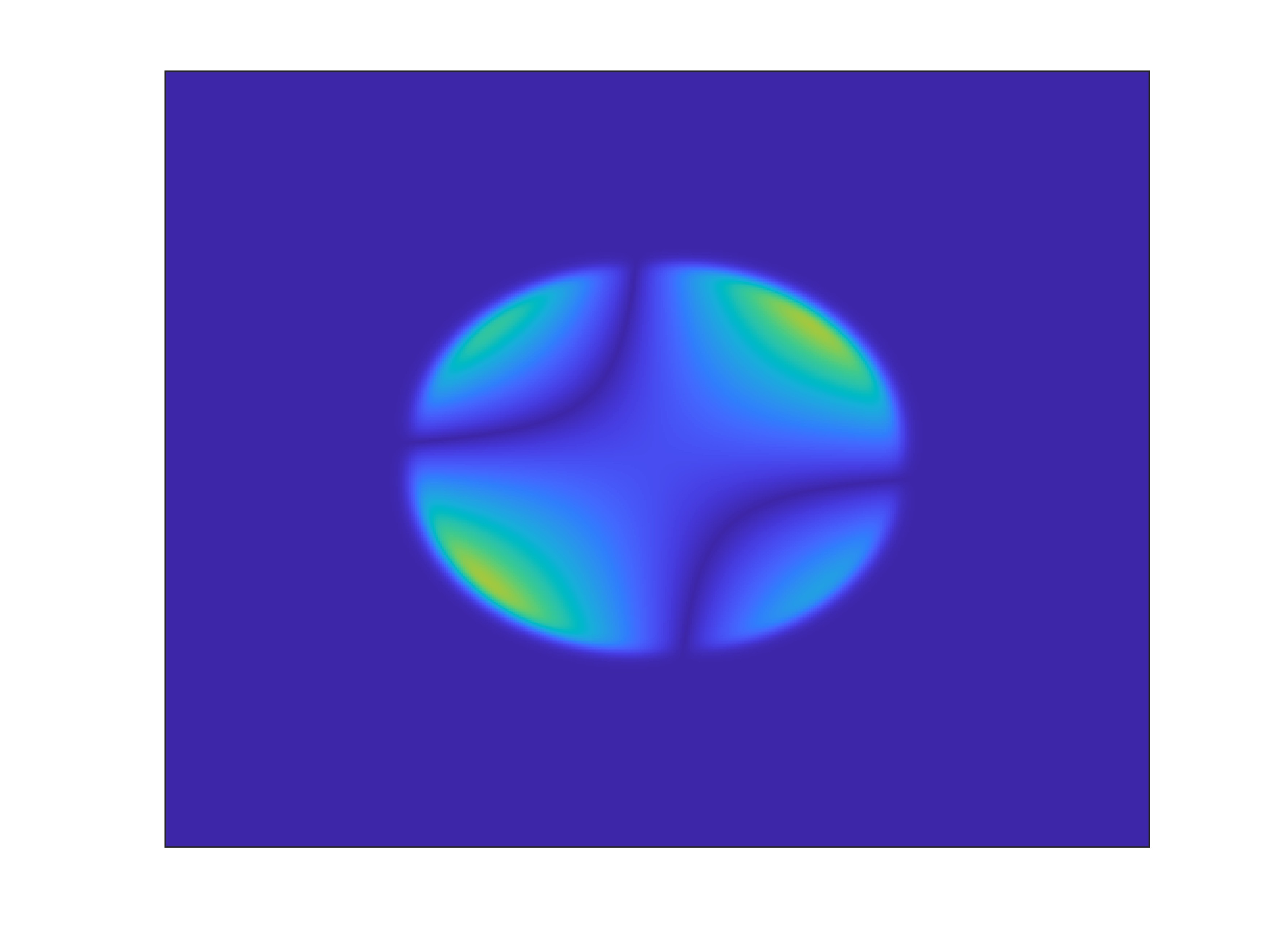}
	}\hspace{-0.6cm}
	\subfigure{
		\label{fig2-3}
		\centering
		\includegraphics[width = 120pt,height=100pt]{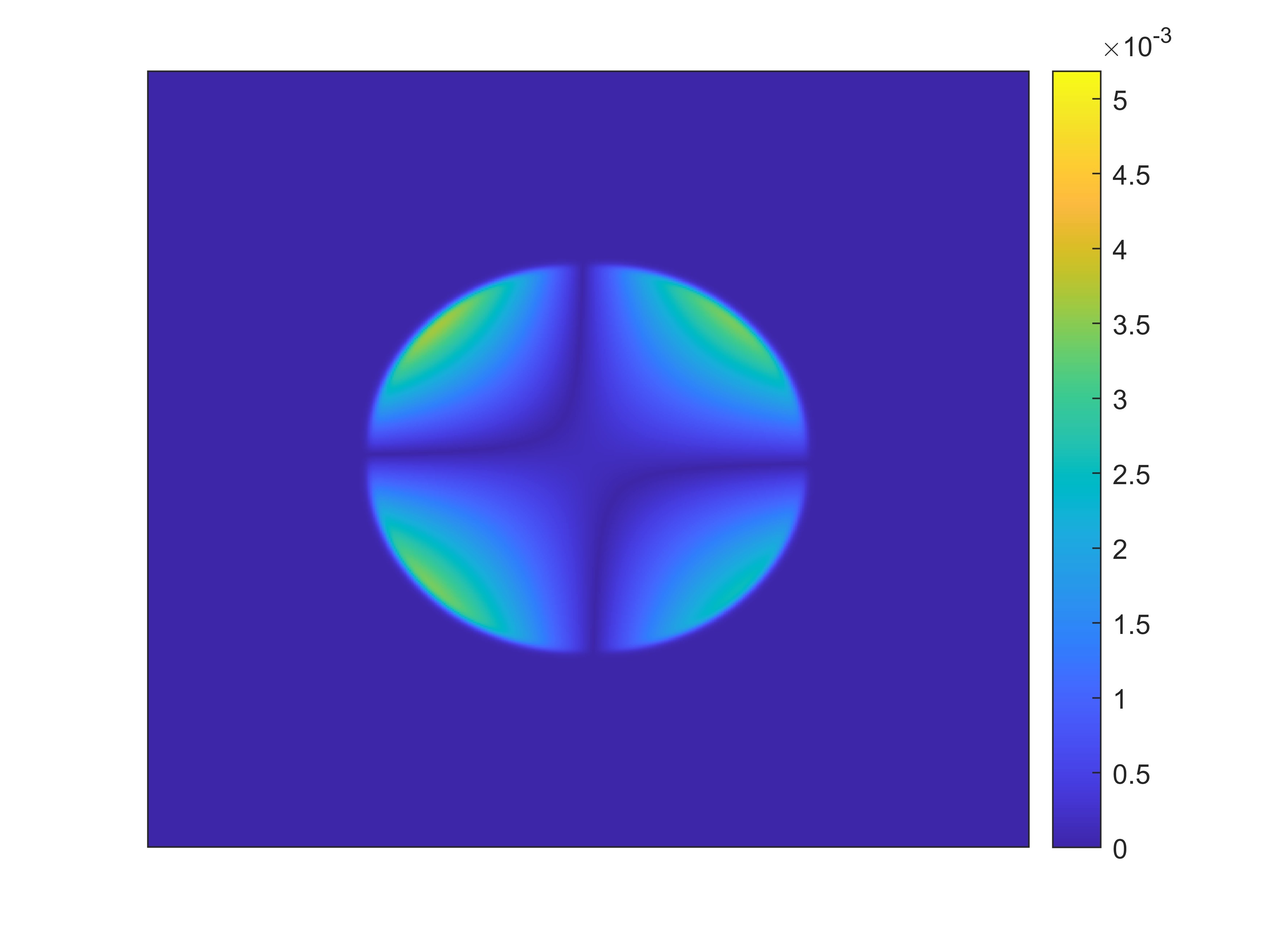}
	}
	
	\subfigure{
		\label{fig3-1}
		\centering
		\includegraphics[width = 110pt,height=100pt]{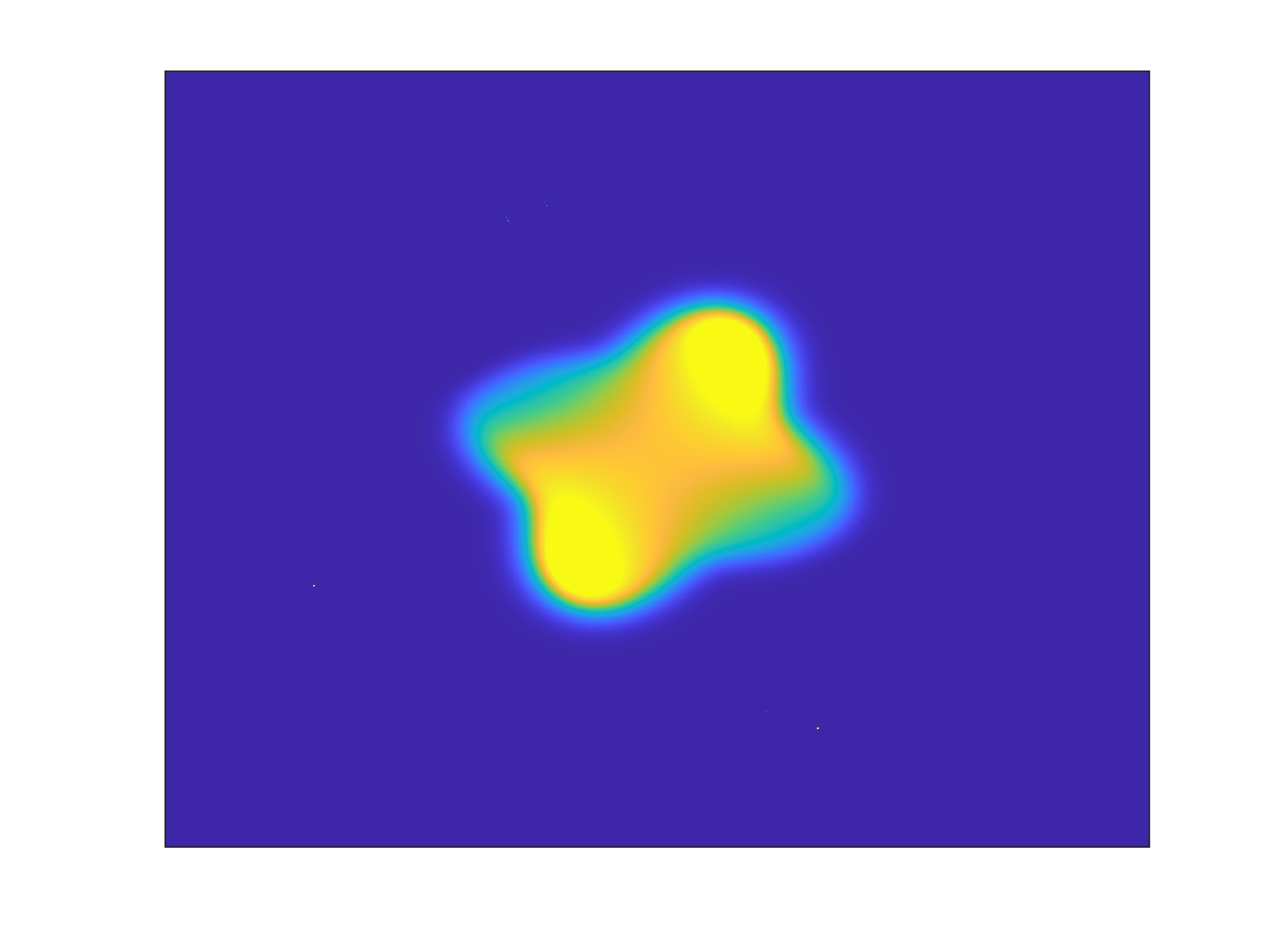}
	}\hspace{-0.6cm}
	\subfigure{
		\label{fig3-2}
		\centering
		\includegraphics[width = 110pt,height=100pt]{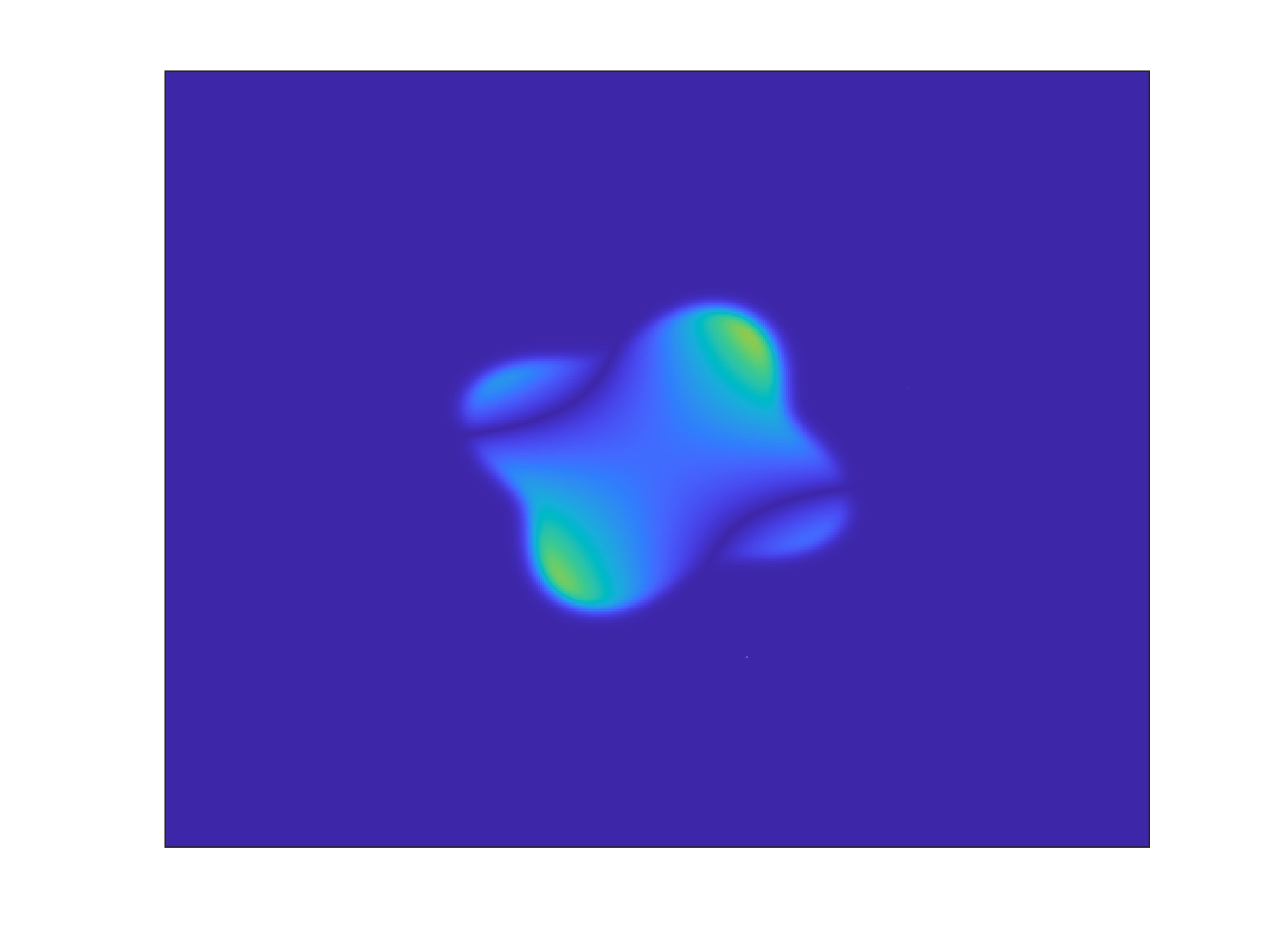}
	}\hspace{-0.6cm}
	\subfigure{
		\label{fig3-3}
		\centering
		\includegraphics[width = 120pt,height=100pt]{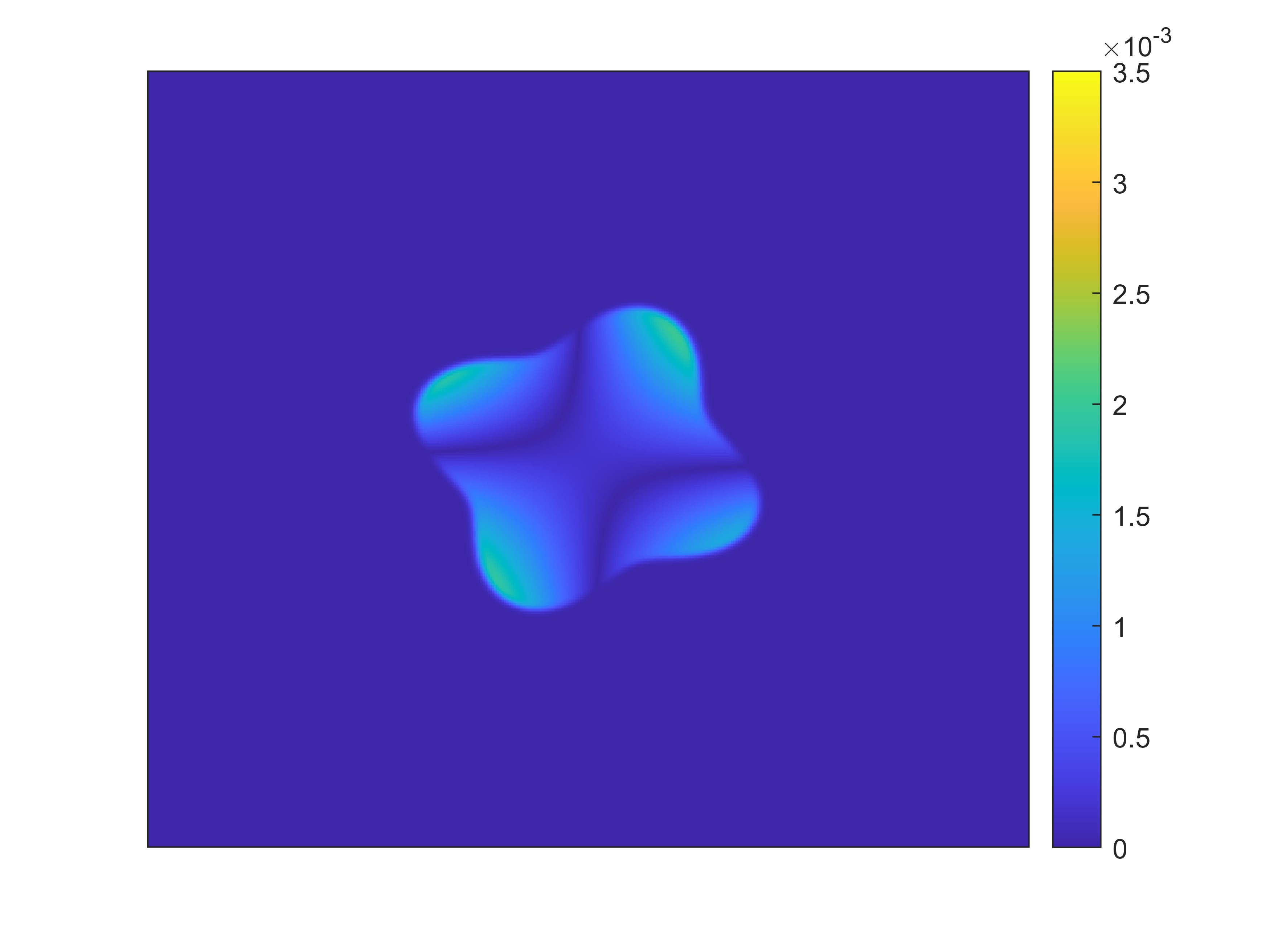}
	}
	\vspace{-0.5cm}
	\caption{The phase structures of the numerical errors at the terminal time $T=0.5$ produced by the DDM approach with interface thickness $\epsilon=1/16,1/32,1/64$ (from left to right) for Example \ref{ex2} in the circular domain (top row) and the flower-shaped domain (bottom row).}
	\label{fig 2}
\end{figure}

\section{Conclusions}\label{conclusion}
In this paper, we studied the convergence and error estimates  of the diffuse domain method for solving a class of second-order parabolic equations  with Neumann boundary conditions  in general irregular domains. We successfully proved optimal error estimates of the diffuse domain solutions in the weighted $L^2$ and $H^1$ norms. Some numerical examples are also presented to  verify  the derived theoretical results. The numerical method and corresponding error analysis framework developed in this paper also naturally enable us to further investigate adaptive finite element methods \cite{Ciarlet1978,BinevDahmen2004,Bieterman1982} for solving PDEs in regions with corner points with solid theoretical support.

\section*{Acknowledgements}

W. Hao is partially supported by NIH grant 1R35GM146894. L. Ju is partially supported by NSF grant DMS-2409634. Y. Xu is supported by Center of Computational Mathematics and Applications (CCMA) at Penn State University.

\section*{Declarations}

The authors have no competing interests to declare that are relevant to the content of this article.

\bibliographystyle{plainnat}
\bibliography{ref}

\end{document}